\documentclass[12pt]{amsart}
\usepackage[normalem]{ulem}
\usepackage{lineno} 
\usepackage{amsmath,amssymb,amsfonts,amsthm,color,latexsym}
\usepackage{xcolor}
\usepackage{float}
\newtheorem{theorem}{Theorem}[section]
\newtheorem{maintheorem}{Theorem}

\newtheorem{secondtheorem}{Theorem}

\newtheorem{maincorollary}{Corollary}

\newtheorem{secondcorollary}{Corollary}

\newtheorem{proposition}[theorem]{Proposition}
\newtheorem{corollary}[theorem]{Corollary}
\newtheorem{lemma}[theorem]{Lemma}
\theoremstyle{definition}
\newtheorem{definition}[theorem]{Definition}
\newtheorem{example}[theorem]{Example}
\newtheorem{remark}[theorem]{Remark}
\numberwithin{equation}{section}

\newcommand{\sing}{\textsf{Sing}}

\newcommand{\F}{\mathcal{F}}
\newcommand{\C}{\mathbb{C}}
\usepackage{graphicx}

\newcommand{\bb}{\hbox{\rm BB}}
\newcommand{\cs}{\text{CS}}

\newcommand{\co}{\mathbb{C}}

\newcommand{\qe}{\mathbb{Q}}
\newcommand{\ze}{\mathbb{Z}}

\newcommand{\cl}[1]{\mathcal{#1}}

\newcommand{\cpt}[1]{\mathbb{C}P^{2}}

\newcommand{\D}{\mathcal{D}}
\newcommand{\B}{\mathcal{B}}
\newcommand{\Ftilde}{\tilde{\mathcal{F}}}

\newcommand{\sep}{\mbox{\rm Sep}}
\newcommand{\tang}{\textsf{tang}}
\newcommand{\iso}{\mbox{Iso}}
\newcommand{\dic}{\mbox{Dic}}
\newcommand{\divv}{\mbox{\rm Div}}
\newcommand{\var}{\mbox{\rm Var}}

\newcommand{\ord}{\textsf{ord}}

\newcommand{\val}{\text{Val}}

\begin{document}
%\linenumbers % para numerar las lneas.
\title[On Milnor and Tjurina numbers of foliations]{On Milnor and Tjurina numbers of foliations}
\date{\today}
\author[A. Fern\'andez-P\'erez]{Arturo Fern\'andez-P\'erez}
\address[Arturo Fern\'{a}ndez P\'erez] {Department of Mathematics. Federal University of Minas Gerais. Av. Ant\^onio Carlos, 6627 
CEP 31270-901\\
Pampulha, Belo Horizonte, Brazil\\
ORCID: 0000-0002-5827-8828}
\email{fernandez@ufmg.br}

\author{Evelia R. Garc\'{i}a Barroso}
\address[Evelia R. Garc\'{i}a Barroso]{Dpto. Matem\'{a}ticas, Estad\'{\i}stica e Investigaci\'on Operativa\\
IMAULL\\
Universidad de La Laguna. Apartado de Correos 456. 38200 La Laguna, Tenerife, Spain\\
ORCID: 0000-0001-7575-2619}
\email{ergarcia@ull.es}

\author{Nancy Saravia-Molina}
\address[Nancy Saravia-Molina]{Dpto. Ciencias - Secci\'{o}n Matem\'{a}ticas, Pontificia Universidad Cat\'{o}lica del Per\'{u}, Av. Universitaria 1801,
San Miguel, Lima 32, Peru\\
ORCID: 0000-0002-2819-8835}
\email{nsaraviam@pucp.edu.pe}


\begin{abstract}

We study the relationship between the Milnor and Tjurina numbers of a singular foliation  $\F$, in the complex plane, with respect to a balanced divisor of separatrices $\mathcal B$ for  $\F$. For that, we associate with $\F$ a new number called the $\chi$-number and we prove that it is a $C^{1}$ invariant for holomorphic foliations. 
We compute  the polar excess number of  $\F$ with respect to a balanced divisor of separatrices $\B$ for $\F$, via the Milnor  number of the foliation,   the multiplicity of some hamiltonian foliations along the separatrices in the support of $\B$ and the $\chi$-number of $\F$.
On the other hand, we generalize, in the plane case and the formal context, the well-known result of G\'omez-Mont given in the holomorphic context, which establishes the equality between the GSV-index of the foliation and the difference  between the Tjurina number of the foliation and the Tjurina number of a set of separatrices of $\F$.  Finally, we state numerical relationships between some classic indices, as Baum-Bott, Camacho-Sad, and variational indices  of a singular foliation  and its  Milnor and Tjurina numbers; and we obtain a bound for the sum of Milnor numbers of the local separatrices of a holomorphic foliation on the complex projective plane.

\end{abstract}
\maketitle

\section{Introduction}

The Milnor and Tjurina numbers are classical invariants in the theory of complex analytic hypersurfaces  with isolated singularity.   The Milnor number of  a holomorphic function germ $f:(\mathbb C^{n}, 0)\longrightarrow (\mathbb C, 0)$ with an isolated singularity is exactly the rank of
the middle homology group of the Milnor fiber of $f$, equal to the number of spheres in its bouquet decomposition.
The notion of the Milnor number for hypersurfaces was introduced  in \cite[Section 7]{Milnor}. The Tjurina number is the dimension of the base space of a semi-universal deformation of the hypersurface. Semi-universal deformations were studied in \cite{Tjurina} and as far as our knowledge reaches, the name Tjurina number appears for the first time in \cite{Greuel}.
The Milnor number  is a topological invariant and the Tjurina number  is an analytic invariant of the singularity. There is an abundant and varied bibliography on the Milnor number of singular hypersurfaces for the classical concept, as well as for the more recent relative concept, known as Bruce-Roberts Milnor number (see for example the recent papers \cite{Bivia} and \cite{Barbosa}). Moreover, the Milnor number is related to polar multiplicities; see, for instance, \cite{Carvalho} and the references therein.

 The Tjurina number has not been studied that well, perhaps because it is an analytical invariant, but in recent years new studies have been published (see for example \cite{Alberichetals}, \cite{Genzmer-Hernandes}, \cite{Wang} and \cite{Almiron}). In the context of singular foliations, the notion of Milnor number appears for the first time with that name in \cite{CLS}, although this notion is found in previous works such as that of  Van den Essen \cite{Van}, where a new proof of Seidenberg's theorem is given. As in the case of hypersurfaces, the Milnor number of a one-dimensional holomorphic foliation is a topological invariant and there is a varied bibliography on it. However, the concept of the Tjurina number of a foliation has been less studied and according to our knowledge, always related to the G\'omez-Mont-Seade-Verjovsky index, after \cite{Gomez}. We emphasize that G\'omez-Mont did not use the terminology of the Tjurina number of a foliation, such a name appears for the first time in \cite[Page 159]{Cano}. In \cite[Corollary 2.7]{Licanic} there is a bound for the Tjurina number of an $\F$-invariant curve $C$ as a function of the Tjurina number of $\F$ with respect to $C$. Some verifications on the relationship between the Milnor  and Tjurina numbers of non-dicritical foliations and their total union of separatrices are collected  in \cite{FS-GB-SM}.

In this work we deepen into the study of the Tjurina number of a singular foliation along 
a reduced curve of separatrices and its relationship with the Milnor number and other invariants and indices associated with the foliation such as the polar excess number, Baum-Bott, Camacho-Sad, and variational indices. Taking into account that a singular foliation could admit infinitely many separatrices -- dicritical foliation -- the Tjurina number of a foliation will be associated with a balanced divisor of separatrices. The notion of balanced divisor of separatrices for a foliation was introduced by Genzmer in \cite{Genzmer} and we recall it in Definition 2.1. This is a geometric object formed by a finite set of separatrices, choosing all isolated separatrices and some separatrices from the ones associated to dicritical components, with weights, possibly negative (those that correspond to poles). In the non-dicritical case, this notion coincides with the total union of separatrices of the foliation. A balanced divisor of separatrices provides a control of the algebraic multiplicity  of the foliation and of its separatrices (see Proposition \ref{prop:Equa-Ba}). In Foliation Theory it is not easy to determine whether an invariant associated with a foliation is topological (remember that, for example, it is not yet known if the algebraic multiplicity is). However,  it has been shown that some indices are analytical invariants, such as the Baum-Bott index (see \cite[Remark 3.2] {Cerveau-Lins Neto}), Camacho-Sad and Variational indices (a good exercise for beginners in the world of foliations). We hope that this article contributes for understanding the dicritical foliations and the invariants associated with them.

\par The paper is organized as follows. In Section \ref{sec:basic tools}, we recall some preliminary notions that are necessary in the paper. 
The first motivation for this work was understanding the relationship between the Milnor and Tjurina numbers of a singular foliation $\F$ and a balanced divisor of separatrices for  $\F$. It is for this purpose that in Section \ref{sect:chi number} we associate a new number $\chi_p(\F)$ to any  (dicritical or non dicritical) singular foliation $\F$ at $(\C^2,p)$. This number is defined in function of the algebraic multiplicities and excess tangency indices of the strict transforms of $\F$ at infinitely near points of $p$. Hence $\chi_p(\F)$ is a $C^{1}$ invariant for holomorphic foliations. We study the properties of the $\chi$-number in Proposition \ref{prop_chi}.
In particular, we prove that it is a nonnegative integer number and any foliation of algebraic multiplicity bigger than 1 is of the second type, a concept that we introduce below, if and only if the $\chi$-number equals zero.

Section \ref{sect:polar} is devoted to the indices associated with a foliation making use of a generic polar curve of it, as the polar intersection and the polar excess. Proposition \ref{prop-B} provides a formula to compute the polar excess number of a foliation with respect to the zero divisor of a reduced balanced divisor of separatrices, and as a consequence, we obtain a  characterization of generalized curve foliations, which generalizes \cite[Proposition 2]{Cano} to the dicritical context. In Proposition \ref{Milnor and chi}, we establish the relationship between the Milnor number of a foliation, the multiplicity of the foliation along the separatrices (of a balanced divisor) and the $\chi$-number of the foliation, generalizing  \cite[Corollary 2]{Cano} to dicritical foliations. As a consequence we give a new proof, using foliations, of the well-known relationship between the Milnor number of a reduced plane curve and the Milnor numbers of its irreducible components (see Proposition \ref{Milnor formula}).
Theorem \ref{Th:Delta-chi} is the main result in this section and one of the main results in this paper. In Theorem \ref{Th:Delta-chi} we compute the polar excess number $\Delta_p(\F,\B)$ of a singular foliation $\F$, with respect to a balanced divisor of separatrices $\B$, via the Milnor  number of the foliation, $\mu_p(\F)$,  the multiplicity of some hamiltonian foliations along the separatrices in the support of $\B$, $\mu_p(d{F}_B,B)$, and the $\chi$-number of $\F$. More precisely

\begin{maintheorem}\label{Th:Delta-chi}
Let $\F$ be a singular foliation at $(\mathbb C^2,p)$ and let $\B=\sum_{B}a_B B$ be a balanced divisor of separatrices for $\F$. Then
\[
\Delta_p(\F,\B)=\mu_p(\F)-\sum_{B}a_B\mu_p(d{F}_B,B)+ \deg(\B)-1-\chi_p(\F),
\]
where for each separatrix $B$, ${F}_B$ is a balanced divisor of separatrices for $\F$ adapted to $B$. Moreover, if $\F$ is  a foliation of the second type, then 
\[
\Delta_p(\F,\B)=\mu_p(\F)-\sum_{B}a_B\mu_p(d{F}_B,B)+ \deg(\B)-1.
\]
\end{maintheorem}

As a consequence, in Corollary \ref{char CG}, we give a new characterization of generalized curve foliations in the non-dicritical case.  

In Section \ref{sect:GSV},  we study the G\'omez-Mont-Seade-Verjovsky index (GSV-index). In particular, 
in Corollary \ref{GSV-zero divisor}, we compute the GSV-index of a foliation $\F$ with respect to the zero divisor of a reduced balanced divisor of separatrices for $\F$.
In Proposition \ref{gsv=delta}, we generalize, to the dicritical case, \cite[Proposition 4]{Cano} which establishes the equality between the GSV-index of a foliation $\F$ (containing perhaps purely formal branches) with respect to an $\F$-invariant curve $C:f(x,y)=0$ and  the intersection numbers of $C$ with generic polar curves of $\F$ and of $df$.

We finish this section  establishing, in Proposition \ref{prop:gsv and multiplicity}, a relationship between the GSV-index and the multiplicity of a foliation along a fixed separatrix. 

In Section \ref{sect:Tjurina}, we introduce the notion of Tjurina number of a singular foliation $\F$ along a reduced curve of separatrices $C$, denoted by $\tau_{p}(\F,C)$.   G\'omez-Mont proved, for a singular foliation with a set of convergent separatrices $C$, that the  difference between the Tjurina number of the foliation and the Tjurina number of $C$, $\tau_{p}(C)$,  equals to the GSV-index (see \cite[Theorem 1]{Gomez}). 
In Proposition \ref{obs_1}, we show that this result also holds, in the formal context, for the Tjurina number of a singular foliation along a reduced curve of separatrices. As a consequence we get the next corollary for non-dicritical foliations:

\begin{maincorollary}
Let $\F$ be a singular foliation at $(\mathbb C^2,p)$. Assume that $\F$ is non-dicritical and $C$ is the total union of separatrices of $\F$. Then
 \[
 \mu_{p}(\F)-\tau_{p}(\F,C)=\mu_{p}(C)-\tau_{p}(C)+\chi_{p}(\F).
 \]
Moreover, if $\F$ is of second type then $\mu_{p}(\F)-\tau_{p}(\F,C)=\mu_{p}(C)-\tau_{p}(C)$.
\end{maincorollary}

  The main result in this section, and another of the main results in the paper,  is Theorem \ref{excess}, where given a balanced divisor of separatrices $\B=\sum_{B}a_{B}B$ of the singular foliation $\F$, we compute the difference of the Milnor of $\F$ and the sum $T_{p}(\F,\B)=\sum_{B}a_{B}\tau_{p}(\F,B)$ of Tjurina numbers of $\F$ along the components of $\B$:

\begin{secondtheorem}\label{excess}
Let $\F$ be a singular foliation at $(\mathbb C^2,p)$ and let $\B=\sum_{B}a_BB$ be a balanced divisor of separatrices for $\F$. Then

\begin{eqnarray*}
\mu_p(\F)- T_{p}(\F,\B)&=&\sum_{B}a_B[ \mu_p(dF_B,B) -\tau_p(B)]-\deg(\B)+1+\chi_p(\F)\\
& & - \sum_{B} a_B[i_p(B,(F_B)_0\setminus B)-i_p(B,(F_B)_{\infty})], 
\end{eqnarray*}

where ${F}_B$ is a balanced divisor of separatrices for $\F$ adapted to $B$.
\end{secondtheorem}

 As a consequence,  we precise this relationship for second type foliations in Corollary \ref{cor:1} and for non-dicritical foliations in:
 
 \begin{secondcorollary}\label{cor:6}
Let $\F$ be a singular foliation at $(\mathbb C^2,p)$. Assume that $\F$ is non-dicritical and $C=\cup_{j=1}^{\ell}C_j$ is the total union of separatrices of $\F$. Then
\[\mu_p(\F)-T_p(\F,C)=\mu_p(C)-\sum_{j=1}^{\ell}\tau_p(C_{j})+\chi_p(\F)-\sum_{j=1}^{\ell}i_p(C_j,C\setminus C_j),\]
where $\tau_p(C_{j})$ is the Tjurina number of $C_{j}.$
\end{secondcorollary}
 
We complete this section with several examples. In particular, in Example \ref{ex:Arturo}, we construct a family of dicritical foliations which are not of the second type. 
We finish  Section \ref{sect:Tjurina} stating numerical relationships between some classic indices, such  as Baum-Bott, Camacho-Sad, and variational indices,  of a singular foliation  and the  Milnor and Tjurina numbers.  Finally, in Section \ref{sect: bound}, we obtain a bound for the sum of Milnor numbers of the local separatrices of a holomorphic foliation on the complex projective plane.

\section{Basic tools}
\label{sec:basic tools}

\par  In order to fix the terminology and the notation, we recall  some basic concepts of local Foliation Theory. Unless we specify otherwise, throughout this text $\F$ denotes a germ of a singular (holomorphic or formal) foliation at $(\C^2,p)$. In local coordinates $(x,y)$ centered at $p$, the foliation is given by a (holomorphic or formal) 1-form
\begin{equation}
\label{oneform}
\omega=P(x,y)dx+Q(x,y)dy,
\end{equation}
or by its dual vector field
\begin{equation}
\label{vectorfield}
v = -Q(x,y)\frac{\partial}{\partial{x}} + P(x,y)\frac{\partial}{\partial{y}},
\end{equation}
where  $P(x,y), Q(x,y)   \in {\mathbb C}[[x,y]]$ are relatively prime, where ${\mathbb C}[[x,y]]$ is the ring of complex formal power series in two variables. 
 The \textit{algebraic multiplicity} $\nu_p(\F)$ is the minimum of the orders $\nu_p(P)$, $\nu_p(Q)$ at $p$ of the coefficients of a local generator of $\F$. 
\par Let $f(x,y)\in  \mathbb{C}[[x,y]]$. We say that $C: f(x,y)=0$  is {\em invariant} by $\F$ or $\F$-\emph{invariant} if $$\omega \wedge d f=(f.h) dx \wedge dy,$$ for some  $h\in \mathbb{C}[[x,y]]$. If $C$ is an irreducible $\F$-invariant curve then we say that $C$ is a {\em separatrix} of $\F$. The separatrix $C$ is analytic if $f$ is convergent.  We denote by $\sep_p(\F)$ the set of all separatrices of $\F$. When $\sep_p(\F)$ is a finite set we say that the foliation $\F$ is {\em non-dicritical} and we call {\it total union of separatrices} of $\F$ to the union of all elements of  $\sep_p(\F)$. Otherwise we say that $\F$ is a {\em dicritical} foliation.

A point $p \in \co^{2}$  is a {\em reduced} or {\em simple} singularity for $\cl{F}$  if the
 linear part ${\rm D} v(p)$ of the vector field $v$ in \eqref{vectorfield}  is non-zero and  has eigenvalues $\lambda_1, \lambda_2 \in \co$ fitting in one of the two following cases:
\begin{enumerate}
\item  $\lambda_1 \lambda_2 \neq 0$ and $\lambda_1 / \lambda_2 \not \in \qe^+$  
(in this case we say that $p$ is a {\em non-degenerate} or {\em complex hyperbolic} singularity). \smallskip
\item $\lambda_1 \neq 0$ and  $\lambda_2= 0$ \ (in this case we say that $p$ is a {\em saddle-node} singularity).
\end{enumerate}

The reduction process of the singularities of a codimension one singular foliation over an ambient space of dimension two was achieved by Seidenberg  \cite{seidenberg}. Van den Essen gave a new proof in \cite{Van}.

A singular foliation $\F$ at $(\C^2,p)$ is a \textit{generalized curve foliation} if it has no saddle-nodes in its reduction process of singularities, that is, the case (1). This concept was defined by Camacho-Lins Neto-Sad \cite[Page 144]{CLS}. In this case, there is a system of coordinates  $(x,y)$ in which $\F$ is induced by the equation
\begin{equation}
\label{non-degenerate}
\omega=x(\lambda_1+a(x,y))dy-y(\lambda_2+b(x,y))dx,
\end{equation}
where $a(x,y),b(x,y)  \in {\mathbb C}[[x,y]]$ are non-units, so that  $\sep_p(\F)$ is formed by two
transversal analytic branches given by $\{x=0\}$ and $\{y=0\}$. In the case (2), up to a formal change of coordinates, the  saddle-node singularity is given by a 1-form of the type
\begin{equation}
\label{saddle-node-formal}
\omega = x^{k+1} dy-y(1 + \lambda x^{k})dx,
\end{equation}
where $\lambda \in \mathbb{C}$ and $k \in \mathbb{Z}_{>0}$ are invariants after formal changes of coordinates (see \cite[Proposition 4.3]{martinetramis}).
The curve $\{x=0\}$   is an analytic separatrix, called {\em strong}, whereas $\{y=0\}$  corresponds to a possibly formal separatrix, called {\em weak} or {\em central}.

%
%Let $C:f(x,y)=0$ be an $\F$-invariant curve, where $f(x,y)\in \C[[x,y]]$ is reduced. Then, as in the convergent case, there are $g,h\in \C[[x,y]]$, with $f$ and $g$ and $f$ and $h$ relatively prime and a $1$-form $\eta$ (see \cite[Lemma 1.1 and its proof]{suwa1995}) such that
%\begin{equation}\label{eq:suwa}
%g\omega=hdf+f\eta.
%\end{equation}

\par Given a  foliation $\F$ at $(\C^2,p)$ we follow \cite[page 1115]{Fernandez} to introduce
 the set $\cl{I}_{p}(\F)$ of \textit{infinitely near points} of $\F$ at $p$. This  is
defined   in a  recursive way along the reduction process of the singularities of $\F$.
We do as follows. Given a sequence of blow-ups $\pi: (\tilde{X},\D) \to
(\co^2,p)$ --- a possibly intermediate step in the reduction process, where  $\D=\pi^{-1}(p)$ and $\tilde{X}$ is the ambient space containing $\D$ --- and a point $q \in \D$ we set:
\begin{itemize}
\item  if $\tilde{\F}$ is $\D$-reduced at $q$, i.e. $q\in\D$ is a reduced singularity for $\tilde{\F}$,  then $\cl{I}_{q}(\tilde{\F}) = \{q\}$;
\item  if  $\tilde{\F}$ is $\D$-singular but not $\D$-reduced at $q$,
we perform a  blow-up $\sigma:( \hat{X}  ,\hat{\D}) \to (\tilde{X},\D)$ at $q$, where $\hat{\D} = \sigma^{-1}(\D) = (\sigma^{*}\D ) \cup D$ and
  $D = \sigma^{-1}(q)$ (here $\sigma^{*}\D$ denotes the strict transform of $\D$).
If $q_{1}, \ldots, q_{\ell}$ are all $\hat{\D}$-singular points of $\hat{\F} = \sigma^{*} \tilde{\F}$ on $D$,
then \[\cl{I}_{q}(\tilde{\F}) = \{q\} \cup \cl{I}_{q_{1}}(\hat{\F}) \cup \ldots \cup \cl{I}_{q_{\ell}}(\hat{\F}).\]
\end{itemize}
In order to simplify notation, we settle that  a  numerical invariant for a foliation $\F$  at $q \in \cl{I}_{p}(\F)$    actually means the same invariant computed
  for the strict transform of $\F$ at $q$.
 \par For a fixed reduction process of singularities $\pi:(\tilde{X},\D)\to(\C^{2},p)$ for $\F$, a component  $D \subset \D$ can be:
\begin{itemize}
\item {\em non-dicritical}, if $D$ is $\tilde{\F}$-invariant. In this case, $D$ contains a finite number of simple singularities. Each non-corner singularity of $D$ carries a separatrix   transversal to $D$, whose projection by $\pi$ is a curve in $\sep_{p}(\cl{F})$. Remember that a corner singularity of $D$ is an intersection point of $D$ with other irreducible component of $\D$.
\item {\em dicritical}, if $D$ is not $\tilde{\F}$-invariant. The reduction process
of singularities gives that $D$ may intersect only non-dicritical components of $\D$ and  $\tilde{\F}$ is everywhere transverse to $D$. The $\pi$-image of a local leaf of $\tilde{\F}$ at each non-corner point of $D$ belongs to $\sep_{p}(\cl{F})$.
\end{itemize}
\par Let $\sigma$ be the blow-up of the reduction process of singularities $\pi$ that generated the component $D\subset\D$. We say that $\sigma$ is \emph{non-dicritical} (respectively \emph{dicritical}) if $D$ is non-dicritical (respectively dicritical).

\par Denote by   $\sep_{p}(D) \subset \sep_{p}(\cl{F})$ the set of separatrices whose strict transform
 by $\pi$ intersects the
component $D \subset \D$. If $B \in \sep_{p}(D)$ with $D$ non-dicritical, $B$ is said to be \textit{isolated}. Otherwise, it is said to be a \textit{dicritical separatrix}.
This determines the   decomposition $\sep_{p}(\cl{F}) = \iso_{p}(\cl{F}) \cup \dic_{p}(\cl{F})$, where notations are self-evident.
The set $\iso_{p}(\cl{F})$  is finite and  contains all   purely
formal separatrices. It   subdivides further    in two classes:
 \textit{weak} separatrices --- those arising from the weak separatrices of saddle-nodes --- and \textit{strong} separatrices --- corresponding to strong separatrices
of saddle-nodes and separatrices of non-degenerate singularities. On the other hand, if  $\dic_{p}(\cl{F})$ is non-empty  then it is an infinite set of analytic separatrices.
Observe that a foliation  $\cl{F}$ is {\em  dicritical}
  when $\sep_{p}(\cl{F})$ is infinite, which is equivalent to saying that $\dic_{p}(\cl{F})$ is non-empty. Otherwise, $\cl{F}$ is {\em non-dicritical}.

Throughout the text, we adopt the language of \textit{divisors} of formal curves.
More specifically, a \textit{divisor of separatrices} for a foliation $\F$ at $(\C^2,p)$ is
a formal sum
\begin{equation}\label{divisor}
\B = \sum_{B \in \text{Sep}_{p}(\F)} a_{B} \cdot B, 
\end{equation}
where the coefficients $a_{B} \in \ze$ are zero except for finitely many $B \in \sep_{p}(\F)$. The set of separatrices $\{B\;:\; a_{B}\neq 0\}$ appearing in \eqref{divisor} is called the \emph{support} of the divisor $\B$ and it is denoted by $\hbox{\rm supp}(\B)$. The \emph{degree} of the divisor $\B$ is by definition $\deg \B=\sum_{B\in \hbox{\rm supp}(\B) }a_{B}$.
We denote by $\divv_{p}(\F)$ the set of all these divisors of separatrices, which turns into a group with the canonical additive structure.
We follow  the usual terminology and notation:
\begin{itemize}
\item $\B \geq 0$ denotes an \textit{effective} divisor, one whose  coefficients are all  non-negative;
\item   there is a unique decomposition $\B = \B_{0} - \B_{\infty}$, where $\B_{0}, \B_{\infty} \geq 0$ are respectively the \textit{zero}
and \textit{pole} divisors of $\B$;
\item the \textit{algebraic multiplicity} of   $\B$ is
$\nu_{p}(\B)=\displaystyle\sum_{B \in \hbox{\rm supp}(\B) } \nu_{p}(B).$
\end{itemize}
Given a    foliation $\F$ and a formal meromorphic equation ${F}(x,y)=\prod_{i=1}^{s}f_{i}(x,y)^{a_{i}}$, whose irreducible components define
 separatrices  $B_i:f_{i}(x,y)=0$ of $\F$ , we associate   the divisor
$ ({F}) = \sum_{i} a_{i} \cdot B_{i}$.
A curve of separatrices ${C}$, associated with a reduced equation ${F}(x,y)$, 
is identified to
the divisor $({F})$ and we write ${C} = ({F})$. Such an effective divisor is named \textit{reduced} if all its coefficients are either $0$ or $1$. In general,  $\B \in \divv_{p}(\F)$  is reduced if both
$\B_{0}$ and $\B_{\infty}$ are reduced divisors.
A divisor $\B$ is said to be \textit{adapted} to a curve of separatrices ${C}$ if $\B_0 -{C} \geq 0$.
\par Following  \cite[page 5]{Genzmer} and \cite[Definition 3.1]{Genzmer-Mol}, we remember the following notion:
\begin{definition}\label{def-balanced-set}
  A \emph{balanced divisor of separatrices}  for $\F$  is a divisor of the form
\[ \B \ = \
\sum_{B\in {\rm Iso}_p(\F)} B+ \sum_{B\in {\rm Dic}_{p}(\F)}\ a_{B}  \cdot B,\]
where the coefficients $a_{B} \in \mathbb{Z}$  are  non-zero except for finitely many $B \in
\dic_{p}(\F)$, and, for each  dicritical  component $D \subset \D$,
 the following equality is respected:
\[\sum_{B \in {\text{Sep}_{p}(D)}}a_{B} = 2- \val(D).\]

  The integer $\val(D)$ stands for the {\em valence} of a component $D \subset \D$ in the reduction process of singularities, that is,  it is the   number of  components of $\D$ intersecting $D$ other from $D$ itself.
\end{definition}

Observe that the notion of balanced divisor of separatrices  generalizes to dicritical foliations the notion of total union of separatrices for non-dicritical foliations.\\

 A balanced divisor $\B=\sum_{B}a_{B}B$ of separatrices of $\F$ is called \textit{primitive} if $a_{B}\in \{-1,1\}$ for any $B\in \hbox{\rm supp}(\B)$.
A \textit{balanced equation of separatrices} is a formal meromorphic function ${F}(x,y)$ whose associated divisor ${C}={C}_0-{C}_{\infty}$ is a balanced divisor. A balanced equation is \textit{reduced}, \textit{primitive} or \textit{adapted} to a curve $C$ if the same is true for the underlying divisor.
\par Remember that the \textit{intersection number} of two formal curves $C$ and $D$ at $(\C^2,p)$ is by definition  \[i_p(C,D)=\dim_{\C}\C[[x,y]]/(g,h),\] 
where $C:g(x,y)=0$, $D: h(x,y)=0$, and $(g,h)$ denotes the ideal generated by $g$ and $h$ in $\C[[x,y]]$. 
The intersection number for formal curves at $(\C^2,p)$ is canonically extended in a bilinear way to divisors of curves.

\par Let $\F$ be a  foliation at $(\C^2,p)$ given by a $1$-form as in (\ref{oneform}), with reduction process $\pi:(\tilde{X},\D)\to (\C^{2},p)$ and let $\tilde{\F} = \pi^{*} \F$ be the strict transform of  $\F$. Denote by $ \sing(\cdot)$ the set of singularities of a foliation.
A saddle-node  singularity $q \in \sing(\Ftilde)$ 
is said to be a \textit{tangent saddle-node} if  its   weak separatrix is contained in the exceptional divisor $\D$, that is, the weak separatrix is an irreducible component of $\D$.

\par  We have the following definition given by Mattei-Salem \cite[D\'efinition 3.1.4]{mattei} to non-dicritical case and used by Genzmer \cite{Genzmer} for arbitrary foliations:
\begin{definition}
\label{def-2ndclass}
 A foliation   is  \textit{in the second class} or is \textit{of second type} if there
are no    tangent saddle-nodes in its reduction process of singularities.
\end{definition}

\par Let $B$ be a separatrix of $\F$ at $p$. Suppose that $\{y=0\}$ is the tangent cone of $B$, then we may choose one of its primitive Puiseux parametrizations 
$\gamma(t)=(t^n,\phi(t))$ at $p$ such that $n=\nu_p(B)$, where $\nu_p(B)$ denotes the algebraic multiplicity of $B$. The \textit{tangency index of $\F$ along $B$ at $p$} (or \textit{weak index} in \cite[page 1114]{Fernandez}) is
\[
\text{Ind}_p(\F,B):=\ord_{t}Q(\gamma(t)).
\]

The tangency index $\text{Ind}_p(\F,B)$ does not depend on the chosen parametrization of $B$  because by properties of the multiplicity number we get the equality $\ord_{t}Q(\gamma(t))=i_p(Q,B)$.
The foliation $\F$ given by the $1$-form defined in \eqref{saddle-node-formal} verifies  $\text{Ind}_p(\F,B)=k+1 > 1$, where $B:\{y=0\}$.

\par The tangency index was defined in \cite[page 159]{CLS}, where the authors denomine it {\it multiplicity of} $\F$ {\it along} $B$ {\it at} $p$ and denoted by $\mu_{\F}(B,p)$. In the same paper the authors define a similar notion, the index with respect to a vector field $Z$ (see page 152) and denote it by  $Ind_{p}(Z/B)$ which coincides with the multiplicity of $\F$ along $B$ at $p$  that we introduce in \eqref{eq_mult} and we denote by $\mu_{p}(\F, B)$. The reader should pay attention to it to avoid confusion. We adopt the notation given by Genzmer  \cite{Genzmer}, instead of the original given in \cite{CLS} since  $\mu_{p}(\F, B)$ resembles a Milnor number, which will be studied in Section \ref{sect:polar}.

\par Given a component $D \subset \D$, we denote by  $\nu(D)$  its multiplicity,    which coincides with
the algebraic multiplicity of a curve $E$ at $(\mathbb{C}^{2},p)$ whose strict transform $\pi^{*} E$ meets $D$ transversally
outside a corner of $\D$. The following invariant is a measure of the existence of tangent saddle-nodes in the reduction process of singularities of a foliation:

\begin{definition}{\rm
 The \emph{tangency excess} of the foliation $\F$ is defined as $\xi_p(\F)=0$, when $p$ is a reduced singularity, and, in the non-reduced case, as the number
\[\xi_{p}(\F)=\sum_{q \in \textsl{SN}(\F) }\nu(D_{q})(\text{Ind}_q(\Ftilde,\tilde B)  -1),\]
where  $\textsl{SN}(\F)$ stands for  the set of tangent saddle-nodes  on $\D$, $\tilde B$ is   the weak separatrix  passing by $q \in \textsl{SN}(\F)$, and $D_q$  is the component of $\D$ containing $\tilde B$. By \eqref{saddle-node-formal}, we observe that 
 $\text{Ind}_q(\Ftilde,\tilde B)=k+1> 1$. 
 }\end{definition}
Remark that $\xi_{p}(\F) \geq 0$ and, by definition, $\xi_{p}(\F) = 0$ if and only if $\textsl{SN}(\F) = \emptyset$, that is, if and only if  $\F$ is of second type.
In several papers (see for example \cite{Fernandez}, \cite{Cabrera-Mol}) the tangency excess of $\F$ is denoted by $\tau_p(\F)$. In this paper, we denote it by $\xi_{p}(\F)$ since we keep the letter $\tau$ for the Tjurina number of a curve or  a foliation.

\par The following proposition proved by Genzmer   (see \cite[Proposition 2.4]{Genzmer}) will be very useful in this paper:

\begin{proposition}
\label{prop:Equa-Ba} Let $\F$ be a singular foliation at $(\C^2,p)$ and $\B$ a balanced divisor of separatrices for $\F$.   Denote by $\nu_{p}(\F)$ and $\nu_{p}(\B)$ their
algebraic multiplicities respectively. Then
\begin{equation}\label{eq_segundo}
\nu_{p}(\F)=\nu_{p}(\B)-1+\xi_{p}(\F).
\end{equation}
Therefore,
\[\nu_{p}(\F)=\nu_{p}(\B)-1\]
if, and only if, $\F$ is a foliation of second type.
\end{proposition}

%\begin{magenta}If $\F$ is a singular foliation at $(\C^2,p)$ and $\B$ is a balanced divisor of separatrices for $\F$, after \cite[Proposition 2.4]{Genzmer}, we get 
%
%\begin{equation}\label{eq_segundo}
%\nu_{p}(\F)=\nu_{p}(\B)-1+\xi_{p}(\F).
%\end{equation}
%Moreover, $\nu_{p}(\F)=\nu_{p}(\B)-1$
%if, and only if, $\F$ is a second type foliation.
%
%
%\end{magenta}
%
%\begin{blue}Hay que revisar donde se usa y si realmente no es mejor dejar la proposici\'on.\end{blue}

Take a primitive parametrization $\gamma:(\C,0)\to(\C^2,p)$, $\gamma(t)=(x(t),y(t))$, of a formal irreducible curve $B: f(x,y)=0$ at $(\C^2,p)$. Note that $B$ is a separatrix of the foliation $\F: \omega=0$ if and only if $\gamma^{*}(\omega)=0$. 
If $B$ is not an $\F$-invariant curve, we define the \textit{tangency order} of $\F$ along $B$ at $p$ as
\begin{equation}
\label{def:tang}
\tang_p(\F,B)=\ord_{t} a(t),
\end{equation}
 where $\gamma^{*}(\omega)=a(t)dt$.  The tangency order does not depend on the choosen parametri\-zation of $B$, since  $\tang_{p}(\F,B)+\mu_p(B)=i_{p}(B, v(f))$ where $v$ is from (\ref{vectorfield}). The tangency index  $i_{p}(B, v(f))$ was introduced in \cite[page 22]{Brunella-libro}.

 \par The behavior under blow-up of the tangency order, in the non-dicritical case, was studied in \cite[equality (4)]{Cano}. The dicritical case is similar. Indeed, if $\F: \omega=0$ is a singular foliation at $(\mathbb C^{2},p)$,  $\tilde{\F}: \tilde \omega=0$ is its strict transform by the blow-up $\sigma$ at $p$ and $B$ is not an $\F$-invariant curve  then we have 
 \[
\tilde \omega=\left \{\begin{array}{ll}
x^{-\nu_p(\F)}\sigma^{*}(\omega)&\text{if $\sigma$ is non-dicritical;}\\
x^{-(\nu_p(\F)+1)}\sigma^{*}(\omega)&\text{if $\sigma$ is dicritical.}\\
\end{array}
\right. 
\]
Evaluating $\tilde \omega$ in a parametrization of the strict transform (by $\sigma$) $\tilde B$ of $B$  and taking orders we get

\begin{equation}\label{tang_index}
\tang_p(\F,B)=
\begin{cases}
\nu_p(\F)\nu_p(B)+\tang_q(\tilde{\F},\tilde B) & \text{if $\sigma$ is non-dicritical};
\medskip \\
(\nu_p(\F)+1)\nu_p(B)+\tang_q(\tilde{\F},\tilde B)  & \text{if $\sigma$ is dicritical},
\end{cases}
\end{equation}
\noindent where $q\in \tilde B\cap \sigma^{-1}(p).$

% \textcolor{red}{ELIMINAR LO QUE SIGUE EN AZUL PUES LO SUSTITUYE LO ANTERIOR EN MAGENTA.}\\
% 
% \begin{blue}
% Since we do not know a reference for the dicritical case, we give the following lemma:
%
%\begin{lemma}
% Let $\F$ be a singular foliation at $(\mathbb C^{2},p)$ and  consider the blow-up $\sigma$ at $p$. If $B$ is not an $\F$-invariant curve  then we have 
%\begin{equation}\label{tang_index}
%\tang_p(\F,B)=
%\begin{cases}
%\nu_p(\F)\nu_p(B)+\tang_q(\tilde{\F},\tilde B) & \text{if $\sigma$ is non-dicritical};
%\medskip \\
%(\nu_p(\F)+1)\nu_p(B)+\tang_q(\tilde{\F},\tilde B)  & \text{if $\sigma$ is dicritical},
%\end{cases}
%\end{equation}
%where $\tilde B$ (respectively $\tilde{\F}$) is the strict transform of  $B$ (respectively $\F$) by $\sigma$ and $q\in \tilde B\cap \sigma^{-1}(p).$
%\end{lemma}
%
%\begin{proof}
% 
%The strict transform $\tilde{\F}$ of the foliation $\F: \omega=0$ is defined by
%\[
%\tilde \omega=\left \{\begin{array}{ll}
%x^{-\nu_p(\F)}\sigma^{*}(\omega)&\text{if $\sigma$ is non-dicritical;}\\
%x^{-(\nu_p(\F)+1)}\sigma^{*}(\omega)&\text{if $\sigma$ is dicritical.}\\
%\end{array}
%\right .
%\]
%The proof follows evaluating $\tilde \omega$ in a parametrization of the strict transform $\tilde B$ of $B$ and taking orders.
%\end{proof}
%\end{blue}
%

In  \cite[Corollary 1]{Cano} it was stablished that $i_p(\B,B)\leq \tang_p(\F,B)+1$ and the equality holds if and only if $\F$ is of second type. In  \cite[Lemma 4.2]{Cabrera-Mol}, the authors improved  \cite[Corollary 1]{Cano} as follows:
\[i_p(\B,B)=\tang_p(\F,B)-\sum_{q \in \cl{I}_{p}(\F)}\nu_q(B)\xi_q(\F)+1,\]
where $\F$ is a singular foliation at $(\C^2,p)$, $\B$ is a balanced divisor of separatrices for $\F$ and $B$ is a branch which is not $\F$-invariant. A proof, similar to the one given in \cite{Cabrera-Mol}, holds for formal and dicritical foliations.

%\textcolor{red}{Quitar lo que sigue hasta el final de la secci\'on si nos quedamos con el texto magenta anterior.}\\
%
%We find the following lemma  in \cite[Lemma 4.2]{Cabrera-Mol} for complex analytic foliations, but it also holds for formal foliations:
%
%\begin{lemma}\label{lemma: index}
%Let $\F$ be a singular foliation at $(\C^2,p)$. Let $\B$ be a balanced divisor of separatrices for $\F$ and $B$ be a branch which is not $\F$-invariant. Then
%\[i_p(\B,B)=\tang_p(\F,B)-\sum_{q \in \cl{I}_{p}(\F)}\nu_q(B)\xi_q(\F)+1,\]
%where the summation runs over all infinitely near points of $\F$ at $p$. 
%\end{lemma}
%
%\par Observe that Lemma \ref{lemma: index} improves \cite[Corollary 1]{Cano} determining explicitly the  difference 
%$(\tang_p(\F,B)+1)-i_p(\B,B),$
%and generalizes the  result to dicritical foliations.

\section{The $\chi$-number of a foliation}

\label{sect:chi number}

\par \par 
For a singular foliation $\F$ at $(\C^2,p)$ we introduce a new number 
\[\chi_p(\F):=\left(\sum_{q\in\cl{I}_p(\F)}\nu_q(\F)\xi_q(\F)\right )-\xi_p(\F).\]

%\begin{magenta}
%The sum $\sum_{q\in\cl{I}_p(\F)}\nu_q(\F)\xi_q(\F)$ coincides with the {\it tangency excess of $\F$ along any irreducible curve} which is not an $\F$-invariant curve, introduced in \cite[equality (8)]{Cabrera-Mol}. Indeed, \end{magenta}
%
%
%\begin{red}
%The sum $\sum_{q\in\cl{I}_p(\F)}\nu_q(\F)\xi_q(\F)$ coincides with the {\it tangency excess of $\F$ along an irreducible curve} which is not an $\F$-invariant curve, introduced in \cite[equality (8)]{Cabrera-Mol}.\end{red}
%\begin{blue}
%En el paper de Cabrera-Mol  ellos ponen $\nu_q(\Gamma)$ en lugar de $\nu_q(\F)$. Tal y como lo definimos nosotros no se ve que dependa de $\Gamma$. Esto hay que aclararlo. En Cabrera-Mol page 11, igualdad (8), la suma que aparece es exactamente $\left(\sum_{q\in\cl{I}_p(\F)}\nu_q(\F)\xi_q(\F)\right )$. Hay que mirar el Lemma 4.2 de ellos y 
%nuestro Lema 4.5 es su Proposici\'on 4.3. Luego Cano-Corral-Mol en la Remark 1 dice que la multiplicidad algebraica de $\Gamma$ es igual a la foliaci\'on.
%\end{blue}
Observe that
\begin{equation}
\label{Xi2}
\chi_p(\F)=\displaystyle\sum_{q\in\cl{I}_p(\F)\setminus\{p\}}\nu_q(\F)\xi_q(\F)+(\nu_p(\F)-1)\xi_p(\F).
\end{equation}

In \cite[Proposition 9.5]{Mol-Rosas} the authors prove that the tangency excess is a $C^{\infty}$ invariant, and after \cite{Rosas} the algebraic multiplicity of a holomorphic foliation is a $C^{1}$ invariant. Hence the $\chi$-number of a holomorphic foliation is a $C^{1}$ invariant. This invariant has the following properties:

%We enclosed some properties of $\chi_p(\F)$ in the following proposition:
\begin{proposition}\label{prop_chi}
Let $\F$ be a singular foliation at $(\C^2,p)$, then we get:
\begin{enumerate}
\item $\chi_p(\F)\geq 0$;
\item if $\F$ is of second type then $\chi_p(\F)=0$;
\item if $\chi_p(\F)=0$, then either $\F$ has algebraic multiplicity 1 at $p$ or $\F$ is of second type;
\item if $\nu_p(\F)>1$, then 
$\chi_p(\F)=0$ if and only if $\F$ is of second type.

\end{enumerate}
\end{proposition}
\begin{proof}
By \eqref{Xi2} we have $\chi_p(\F)=\beta+(\nu_p(\F)-1)\xi_p(\F)$, where 
$\beta=\sum_{q\in\cl{I}_p(\F)\setminus\{p\}}\nu_q(\F)\xi_q(\F).$
Clearly, $\chi_p(\F)\geq 0$, since it is the sum of two nonnegative numbers. Now, if $\F$ is of second type, we get $\xi_q(\F)=0$, for all $q\in\cl{I}_p(\F)$, which implies that $\chi_p(\F)=0$. 
On the other hand, if $\chi_p(\F)=0$ then $\beta=0$ and $(\nu_p(\F)-1)\xi_p(\F)=0$. This finishes the proof of (3). Item (4) is an immediate consequence of (2) and (3). 
\end{proof}

%\[\beta=(1-\nu_p(\F))\xi_p(\F).\]
%We assert that $\beta=0$. In fact, suppose by contradiction that $\beta>0$. Since $\xi_p(\F)\geq 0$ we deduce that $1-\nu_p(\F)>0$ and hence $\nu_p(\F)=0$, which is a contradiction because $\F$ is a non-trivial singular foliation. This proves the assertion on $\beta$. 
%Therefore $\beta=0$, which implies that either $\nu_p(\F)=1$, i.e., $\F$ has algebraic multiplicity 1 or $\xi_p(\F)=0$, that is, $\F$ is of second type at $p$. This proves (3). Item (4) is an immediately consequence of (2) and (3). 
%\end{red}
%\begin{blue} El referee propone simplificar la prueba, dice que $\chi_p(\F)=0$ implica $(\nu_p(\F)-1)\xi_p(\F)=0$.
%\end{blue}
%\end{proof}

\begin{remark}
Let $\omega=4xydx+(y-2x^2)dy$ be a $1$-form. Observe that the foliation $\F:\omega=0$ at $(\mathbb{C}^2,0)$ is not of second type, its algebraic multiplicity is one but $\chi_0(\F)=1\neq 0$.
\end{remark}

\section{Polar intersection, polar excess and Milnor  numbers}
\label{sect:polar}

Let $\omega=P(x,y)dx+Q(x,y)dy$ be a 1-form, where $P(x,y),Q(x,y)\in \mathbb C[[x,y]]$.
If $\F: \omega=0$ is a singular (analytic or formal) foliation then {\it the polar curve} of $\F$  at $(\C^2,p)$ with respect to a point $(a:b)$ of the complex projective line $\mathbb P^1(\mathbb C)$ is the (analytic or formal)  curve $\mathcal{P}^{\F}_{(a:b)}: aP(x,y)+bQ(x,y)=0$. 
Observe that when $\F$ is the hamiltonian foliation associated to a function $f$, the polar curve of $\F$ coincide with the classical polar curve of $f$ in the direction $(a:b)$ studied by Teissier \cite{Teissier_inv} and others.  According to the general results on equisingularity (see \cite{Zariski} and \cite{Teissier_polar}), there exists a Zariski open $U$ of the space $\mathbb P^1(\mathbb C)$ of projection directions such that for $(a:b)$  the polar curves are all equisingular. Any element of this set is called {\it generic polar curve} of the foliation $\F$ and we will denote it by $\mathcal{P}^{\F}$. 

\noindent We borrow from \cite[Section 4]{Genzmer-Mol} the notion of polar curve of a meromorphic 1-form: let $\eta=\frac{\omega}{H(x,y)}$ be a  meromorphic 1-form,  where
$\omega=P(x,y)dx+Q(x,y)dy$ with $P(x,y),Q(x,y),H(x,y)\in \mathbb C[[x,y]]$. The {\it polar curve}  of  $\eta$ at $(\C^2,p)$ with respect to  $(a:b)\in \mathbb P^1(\mathbb C)$ is the divisor $\mathcal{P}^{\eta}_{(a:b)}$ with formal meromorphic equation \[\frac{aP(x,y)+bQ(x,y)}{H}=0.\] 
A polar curve $\frac{aP(x,y)+bQ(x,y)}{H}=0 $ of a  meromorphic 1-form  $\frac{\omega}{H(x,y)}$ will be generic if the polar curve $aP(x,y)+bQ(x,y)=0$ is a generic polar curve of the foliation defined by the $1$-form $\omega$. \par 

Let $B: h(x,y)=0$ be a separatrix  of a singular foliation $\F$. The {\it polar intersection number} of $\F$ with respect to $B$ is the intersection number $i_p({\mathcal P}^{\F},B)$.
%\begin{blue}El referee piensa que el siguiente lema es cierto incluso si $B$ no es  aislada suponiendo que $B$ no es componente de los divisores de polos de $F$ ni  de $G$.\end{blue}

\begin{lemma}
\label{lema: db}
Let $B: h(x,y)=0$ be a separatrix  of a singular foliation $\F$ at $(\mathbb C^2,p)$ and consider ${F}_{B}$ and ${G}_{B}$ two balanced divisors of separatrices for $\F$ adapted to $B$. Then
\[
i_p(\mathcal{P}^{d{F}_{B}},B)=i_p(\mathcal{P}^{d{G}_{B}},B).
\]
\end{lemma}
\begin{proof}
Put $ F_{B}=\frac{h\cdot \hat f \cdot g_{1}\cdots g_{l}}{\phi_{1}\cdots \phi_{m}}$ and
$ G_{B}=\frac{h\cdot \hat f \cdot h_{1}\cdots h_{s}}{\psi_{1} \cdots \psi_{r}}$, where $g_{i}(x,y)=0$, $h_{i}(x,y)=0$, $\phi_{i}(x,y)=0$ and $\psi_{i}(x,y)=0$ are dicritical separatrices of $\F$, $\hat f(x,y)=0$ defines the reduced curve which is the union of all isolated separatrices of $\F$ except perhaps $h(x,y)=0$ when this is also isolated. We get

\[d F_{B}=\frac{\phi(\mathcal P \cdot h+\hat f \cdot g_{1}\cdots g_{l}\cdot dh)-h\cdot \hat f \cdot g_{1}\cdots g_{l}d(\phi)}{\phi^{2}}
\]
\noindent and
\[
d G_{B}=\frac{\psi(\mathcal Q \cdot h+\hat f\cdot h_{1}\cdots h_{s}\cdot dh)-h\cdot \hat f \cdot h_{1}\cdots h_{s}d(\psi)}{\psi^{2}},
\]
\noindent where $\phi=\phi_{1}\cdots \phi_{m}$, $\psi=\psi_{1}\cdots \psi_{r}$, $\mathcal P=d(\hat f \cdot g_{1}\cdots g_{l})=:\mathcal P_{1}dx+\mathcal P_{2}dy$ and $\mathcal Q=d( \hat f \cdot h_{1}\cdots h_{s})=:\mathcal Q_{1}dx+\mathcal Q_{2}dy$. Put $u=\hat f \cdot  g_{1}\cdots g_{l}$ and $v=\hat f \cdot  h_{1}\cdots h_{s}$.   Hence
\[
\mathcal{P}^{d\hat{F}_{B}}=\frac{[\phi(\mathcal P_{1}\cdot h+u \cdot  \partial_x{h})- h u \cdot  \partial_x{\phi}]a+ 
[\phi(\mathcal P_{2}\cdot h+u \cdot \partial_y{h})- h u \cdot  \partial_y{\phi}]b}{\phi^2},
\]
\noindent and
\[
\mathcal{P}^{d\hat{G}_{B}}=\frac{[\psi(\mathcal Q_{1}\cdot h+v\cdot   \partial_x{h})- h v \cdot  \partial_x{\psi}]a+ 
[\psi(\mathcal Q_{2}\cdot h+v \cdot  \partial_y{h})- h v \cdot  \partial_y{\psi}]b}{\psi^{2}}.
\]
So
\begin{eqnarray}
i_p(\mathcal{P}^{d\hat{F}_{B}},B)&=&i_p\left(\hat f \cdot  g_{1}\cdots g_{l}\cdot  \left(a \partial_x{h}+b \partial_y{h}\right),h\right)-i_{p}(\phi_{1}\cdots \phi_{m},h) \nonumber \\ 
&=& i_p\left(\hat f \cdot  \left(a \partial_x{h}+b \partial_y{h}\right),h\right)+i_p( g_{1}\cdots g_{l},h)-i_{p}(\phi_{1}\cdots \phi_{m},h),
\end{eqnarray}
\noindent and
\begin{eqnarray}
i_p(\mathcal{P}^{d\hat{G}_{B}},B)& = & i_p\left(\hat f \cdot  h_{1}\cdots h_{s}\cdot  \left(a \partial_x{h}+b \partial_y{h}\right),h\right)
-i_{p}(\psi_{1}\cdots \psi_{r},h) \nonumber \\
& = & i_p\left(\hat f \cdot  \left(a \partial_x{h}+b \partial_y{h}\right),h\right)+i_p( h_{1}\cdots h_{s},h)-i_{p}(\psi_{1}\cdots \psi_{r},h).
\end{eqnarray}
\noindent We claim that $i_p( g_{1}\cdots g_{l},h)-i_{p}(\phi_{1}\cdots \phi_{m},h)=i_p( h_{1}\cdots h_{s},h)-i_{p}(\psi_{1}\cdots \psi_{r},h)$. Indeed, if every dicritical separatrix of $\F$ is smooth and transversal to any isolated separatrix, then using properties of the intersection multiplicity  we have
\begin{eqnarray}
i_p( g_{1}\cdots g_{l},h)-i_{p}(\phi_{1}\cdots \phi_{m},h)& = &\nu_{p}(h)\left [\sum_{j=1}^{l}\nu_{p}(g_{j})-\sum_{j=1}^{m}\nu_{p}(\phi_{j})\right ] \nonumber\\
& = &\nu_{p}(h)\left [\sum_{j=1}^{s}\nu_{p}(h_{j})-\sum_{j=1}^{r}\nu_{p}(\psi_{j})\right ] \label{eq: db}\\
& = & i_p( h_{1}\cdots h_{s},h)-i_{p}(\psi_{1}\cdots \psi_{r},h), \nonumber
\end{eqnarray}

\noindent where the equality \eqref{eq: db} holds since ${F}_{B}$ and ${G}_{B}$ are two balanced divisors of separatrices for $\F$. 

In the general case, after the reduction of singularities of the foliation we can suppose that the strict transform of every dicritical separatrix of $\F$ is smooth and transversal to any strict transform of every isolated separatrix.
We finish the proof using Noether formula.
\end{proof}

\par Lemma \ref{lema: db} allows us to define the {\it polar excess number} of a singular foliation $\F$ at $(\mathbb C^{2},p)$ with respect to a separatrix $B$ of  $\F$  as
 \begin{equation}
\label{EN:1}
\Delta_p(\F,B):=i_p(\mathcal{P}^{\F},B)-i_p(\mathcal{P}^{d{F}_{B}},B),
\end{equation}
where ${F}_{B}$ is any balanced divisor of separatrices for $\F$ adapted to $B$.
On the other hand, if the foliation $\F$ is non-dicritical  then it is enough to consider ${F}_{B}$ as 
the total union of the separatrices of $\F$.
\par Using properties of the intersection number we extend the definitions of  polar intersection and polar excess numbers to any divisor   ${\mathcal B}:=\sum_B a_B B$ of separatrices of $\F$ in the following way:
\[
 i_p({\mathcal P}^{\F},{\mathcal B})=\sum_{B}a_B i_p({\mathcal P}^{\F},B)
\]
\noindent and
 \begin{equation}
\label{END:1}
\Delta_p(\F,{\mathcal B})=\sum_{B}a_B \Delta_p(\F,B)=\Delta_p(\F,{\mathcal B}_0)-\Delta_p(\F,{\mathcal B}_{\infty}).
\end{equation}
If $\B$ is a primitive divisor then the difference $\Delta_p(\F,{\mathcal B}_0)-\Delta_p(\F,{\mathcal B}_{\infty})$ is independent of the choosen primitive balanced divisor of separatrices for $\F$ (see \cite[Section 3.6, page 1123]{Fernandez}).

By \cite[Proposition 4.6]{Genzmer-Mol}, $\Delta_p(\F,B)$ is a non-negative integer number  for any irreducible component $B$ of $\B_{0}$. As a consequence  $\Delta_p(\F,\B)$ is also a non-negative integer number for any effective divisor of separatrices $\B$.

\par The \textit{Milnor number} $\mu_{p}(\F)$ of the foliation $\F$ at $p$ given by the $1$-form $\omega=P(x,y)dx+Q(x,y)dy$ is defined by 
\[ \mu_p(\F)=i_p(P,Q).\] 
Remember that we consider  $P$ and $Q$  coprime, so $\mu_p(\F)$ is a non negative integer. In \cite[Theorem A]{CLS} it was proved that the Milnor number of a foliation is  a topological invariant. 
\par On the other hand, the \textit{Milnor number} $\mu_{p}(C)$ at  $p$ of a plane curve $C$ (non necessary irreducible) with equation $f(x,y)=0$  is 
\[ \mu_p(C)=i_p\left(\partial_{x} f,\partial_{y} f\right).\] 
Observe that $\mu_p(C)$ is finite if and only if $f$ has not multiple factors, that is, the curve $C$ is reduced.\\

\par  Generalized curve foliations have a property of minimization of Milnor numbers and are characterized, in the non-dicritical case, by several authors, see for instance \cite[Proposition 7]{index} and \cite[Th\'eor\`eme 3.3]{Cav-Le}. Recently, in \cite[Theorem A]{Genzmer-Mol}, the authors have characterized singular generalized curve foliations in terms of its polar excess number as follows: let $\F$ be a singular foliation at $(\C^2,p)$ and let $\B=\B_0-\B_{\infty}$ be a balanced divisor of separatrices for $\F$, then $\Delta_{p}(\F,\B_{0})=0$ if and only if $\F$ is a generalized curve foliation. \\

\par The following proposition provides a formula to compute the polar excess number of a foliation with respect to the zero divisor of a reduced balanced divisor of separatrices.

\begin{proposition}\label{prop-B}
Let $\F$ be a  singular foliation at $(\C^2,p)$ and let $\B=\B_0-\B_{\infty}$ be a reduced balanced divisor of separatrices for $\F$. Then
\[
\Delta_p(\F,\B_{0})=i_{p}(\mathcal P^{\F},\B_{0})+i_{p}(\B_{0},\B_{\infty})-\mu_p(\B_{0})-\nu_p(\B_{0})+1.
\]

\noindent Moreover, $\F$ is a generalized curve foliation if and only if 
\[
i_{p}(\mathcal P^{\F},\B_{0})=\mu_p(\B_{0})+\nu_p(\B_{0})-i_{p}(\B_{0},\B_{\infty})-1.
\]

\end{proposition}
\begin{proof}
Let $\omega=P(x,y)dx+Q(x,y)dy$  be a 1-form inducing $\F$ and $f(x,y)=0$  and $g(x,y)=0$ be the reduced equations of $\B_0$ and $\B_{\infty}$ respectively. Since $\B$ is a set of separatrices of $\F$ adapted to $\B_0$, by \eqref{EN:1} and the definition of the polar curve of $d(f/g)$ we have
\begin{eqnarray*}
\Delta_p(\F,\B_{0})&=& i_{p}(\mathcal P^{\F},\B_{0})-i_{p}(\mathcal P^{d(f/g)},\B_{0})\\
&=& i_{p}(\mathcal P^{\F},\B_{0})-i_p\left(\frac{(g\partial_{x}f-f\partial_{x}g)a+(g\partial_{y}f-f\partial_{y}g)b}{g^{2}},f\right)\\
&=& i_{p}(\mathcal P^{\F},\B_{0})-i_p\left(\frac{g(a\partial_{x}f+b\partial_{y}f)-f(a\partial_{x}g+b\partial_{y}g)}{g^{2}},f\right)
\end{eqnarray*}
where $(a:b)\in\mathbb{P}^1$. By applying properties on intersection numbers and Teissier's Proposition \cite[Chapter II, Proposition 1.2]{Teissier}, we get  
\begin{eqnarray*}
\Delta_p(\F,\B_{0})&=& i_{p}(\mathcal P^{\F},\B_{0})-i_p(a\partial_{x}f+b\partial_{y}f,f)+i_{p}(g,f)\\
&=&  i_{p}(\mathcal P^{\F},\B_{0})+i_p(\B_0,\B_{\infty})-\mu_p(\B_{0})-\nu_p(\B_{0})+1.
\end{eqnarray*}

The second part of the proposition follows from the first part and the characterization of generalized curve foliations given in  \cite[Theorem A]{Genzmer-Mol}.
\end{proof}
The second part of Proposition \ref{prop-B} generalizes \cite[Proposition 2]{Cano} as it does not have the restriction of non-dicriticality.

We give a numerical illustration of Proposition \ref{prop-B}.
\begin{example} 
  Let $\F$ be the foliation at $(\C^2,0)$ defined by 
$\omega=xdy-ydx.$
Observe that $\F$ is a dicritical generalized curve foliation called the {\it radial foliation}. It has only one dicritical component whose valence is $0$ in its reduction process of singularities. Thus $\B=(x)+(y)+(x-y)-(x+y)$ is a reduced balanced divisor of separatrices for $\F$, where $\B_0: xy(x-y)=0$ and $\B_{\infty}: x+y=0$. We get $\mu_0(\B_0)=4$, $\nu_0(\B_0)=3$,
$i_0(\B_0,\B_{\infty})=3$ and $i_0(\mathcal P^{\F},\B_{0})=3.$ Observe that if we consider the reduced balanced divisor of separatrices $\B=(x)+(y)$ for $\F$ then $\B_{\infty}$ is a unit, so  $i_{p}(\B_{0},\B_{\infty})=0$ and now  $\mu_0(\B_0)=1$, $\nu_0(\B_0)=2$. On the other hand, if we consider the foliation $\F$ with a saddle-node (so $\F$ is not a generalized curve foliation) and equation as in \eqref{saddle-node-formal} we get 
again $\B=(x)+(y)$ but  $i_{p}(\mathcal P^{\F},\B_{0})=k+2\neq 2$.

\end{example}

\begin{lemma}\label{index_2}
Let $\F$ be a singular foliation at $(\mathbb C^{2},p)$ and let $\B$ be a balanced divisor of separatrices for $\F$. Then
\[i_p(\mathcal{P}^{\F},\B)=\mu_p(\F)+\nu_p(\F)-\sum_{q\in\cl{I}_{p}(\F)}\nu_q(\F)\xi_q(\F),\]
where the summation runs over all infinitely near points of $\F$ at $p$.
\end{lemma}

\begin{proof}
Let $\mathcal{P}^{\F}$ be a generic polar. Denote by $\Gamma(\mathcal{P}^{\F})$ the set of irreducible components of $\mathcal{P}^{\F}$. It follows from \cite[Lemma 4.2]{Cabrera-Mol} that
\begin{eqnarray}
i_p(\mathcal{P}^{\F},\B)&=&\sum_{A\in\Gamma(\mathcal{P}^{\F})}i_p(A,\B)\nonumber\\
&=&\sum_{A\in\Gamma(\mathcal{P}^{\F})}\left(\tang_p(\F,A)-\sum_{q\in\cl{I}_{p}(\F)}\nu_q(A)\xi_q(\F)+1\right)\nonumber\\
&=&\sum_{A\in\Gamma(\mathcal{P}^{\F})}(\tang_p(\F,A)+1)-\sum_{A\in\Gamma(\mathcal{P}^{\F})}\left(\sum_{q\in\cl{I}_{p}(\F)}\nu_q(A)\xi_q(\F)\right)\nonumber\\
&=&\sum_{A\in\Gamma(\mathcal{P}^{\F})}(\tang_p(\F,A)+1)-\sum_{q\in\cl{I}_{p}(\F)}\left(\sum_{A\in\Gamma(\mathcal{P}^{\F})}\nu_q(A)\right)\xi_q(\F)\nonumber\\
&=&\sum_{A\in\Gamma(\mathcal{P}^{\F})}(\tang_p(\F,A)+1)-\sum_{q\in\cl{I}_{p}(\F)}\nu_q(\mathcal{P}^{\F})\xi_q(\F).\label{eq_3}
\end{eqnarray}
According to the proof of \cite[Proposition 2]{Cano}, we have 
\[\sum_{A\in\Gamma(\mathcal{P}^{\F})}(\tang_p(\F,A)+1)=\mu_p(\F)+\nu_p(\F),\] and by \cite[Remark 1]{Cano}
 $\nu_q(\mathcal{P}^{\F})=\nu_q(\F)$. Substituting these terms in the equation (\ref{eq_3}), we obtain
\[i_p(\mathcal{P}^{\F},\B)=\mu_p(\F)+\nu_p(\F)-\sum_{q\in\cl{I}_{p}(\F)}\nu_q(\F)\xi_q(\F).\]

\end{proof}

Lemma \ref{index_2} improves \cite[Proposition 2]{Cano} determining explicitly  the difference between the polar intersection number with respect to a balanced divisor of separatrices $\B$ and the  sum of the Milnor  number and the algebraic multiplicity of the foliation  $\F$. It also generalizes the result to dicritical foliations. On the other hand
comparing Lemma \ref{index_2} and  \cite[Proposition 4.3]{Cabrera-Mol}   (proved for complex analytic foliations, but it also holds for formal foliations) we conclude that
\[
\sum_{q\in\cl{I}_{p}(\F)}\nu_q(\F)\xi_q(\F)=\sum_{q\in\cl{I}_{p}(\F)}\nu_q(B)\xi_q(\F),
\]
for any $B$ which is not an $\F$-invariant curve. Hence the sum $\sum_{q\in\cl{I}_p(\F)}\nu_q(\F)\xi_q(\F)$ coincides with the {\it tangency excess of $\F$ along any irreducible curve} which is not an $\F$-invariant curve, introduced in \cite[equality (8)]{Cabrera-Mol}. In particular, after the definition of the $\chi$-number, this tangency excess equals $\chi_p(\F)+\xi_p(\F)$. \\

\par Let $\F$ be a singular foliation at $(\C^2,p)$ induced by the vector field $v$ and $B$ be a separatrix of $\F$. Let $\gamma:(\C,0)\to(\C^2,p)$ be a primitive parametrization of $B$, we can consider the \textit{multiplicity of $\F$ along $B$ at $p$} defined by 
\begin{equation}\label{eq_mult}
\mu_p(\F,B)=\ord_t \theta(t),
\end{equation}
where $\theta(t)$ is the unique vector field at $(\C,0)$ such that 
$\gamma_{*} \theta(t)=v\circ\gamma(t)$, see for instance \cite[page 159]{CLS}. If $\omega=P(x,y)dx+Q(x,y)dy$ is a 1-form inducing $\F$ and $\gamma(t)=(x(t),y(t))$, we get 
\begin{equation}\label{mult_2}
\theta(t)=
\begin{cases}
-\frac{Q(\gamma(t))}{x'(t)} & \text{if $x(t)\neq 0$}
\medskip \\
 \frac{P(\gamma(t))}{y'(t)} & \text{if $y(t)\neq 0$}.
\end{cases}
\end{equation}
Hence, by taking orders we obtain
\begin{equation}
\mu_p(\F,B)=
\begin{cases}
\ord_t Q(\gamma(t))-\ord_t x(t)+1 & \text{if $x(t)\neq 0$};
\medskip \\
 \ord_t P(\gamma(t))-\ord_t y(t)+1 & \text{if $y(t)\neq 0$}.
\end{cases}
\end{equation}

The following proposition has been proved in \cite[Proposition 1]{Cano} for non-dicritical foliations, but we may check that it is valid for dicritical foliations.

\begin{proposition}\label{cano_pro}
Consider a separatrix $B$ of a singular (holomorphic or formal) foliation $\F$ at $(\C^2,p)$. We have 
\[i_p(\mathcal{P}^\F,B)=\mu_p(\F,B)+\nu_p(B)-1.\]
\end{proposition}

\begin{remark}\label{nonnegative}
Let $\F$ be a singular foliation at $(\mathbb C^2,p)$. Assume that $\F$ is non-dicritical and $C=\cup_{j=1}^{\ell}C_j$ is the total union of separatrices of $\F$. Applying Proposition \ref{cano_pro} to $\F$ and $df$, where $C:f(x,y)=0$, we get $\Delta(\F,C_{j})=\mu_p(\F,C_j)-\mu_p(df,C_j),$ for $j=1, \ldots, \ell$.
Since $\Delta(\F,C_{j})\geq 0$ we have $\mu_p(\F,C_j)\geq \mu_p(df,C_j)$ for any separatrix $C_{j}$ of  $\F$.
\end{remark}

As a consequence of Lemma \ref{index_2} and Proposition \ref{cano_pro} we obtain a generalization of \cite[Corollary 2]{Cano}.
\begin{proposition}
\label{Milnor and chi}
Let $\F$ be a singular (holomorphic or formal) foliation at $(\mathbb C^{2},p)$ and let  $\B=\sum_{B}a_BB$ be a balanced divisor for separatrices of $\F$. We have
\[\mu_p(\F)=\sum_{B}a_B\mu_p(\F,B)+\chi_p(\F)-\deg(\B)+1.\]
\end{proposition}
\begin{proof}
By summing up polar intersection numbers over all irreducible components of $\B$ and applying Proposition \ref{cano_pro}, we get 
\begin{eqnarray*}
i_p(\mathcal{P}^{\F},\B)&=&\sum_{B}a_Bi_p(\mathcal{P}^{\F},B)=\sum_{B}a_B(\mu_p(\F,B)+\nu_p(B)-1)\\
&=& \sum_{B}a_B\mu_p(\F,B)+\left(\sum_{B}a_B\nu_p(B)\right)-\deg(\B)\\
&=& \sum_{B}a_B\mu_p(\F,B)+\nu_p(\B)-\deg(\B).
\end{eqnarray*}

From Lemma \ref{index_2}, Proposition \ref{prop:Equa-Ba}  and the definition of the $\chi$-number of $\F$ we get

\begin{eqnarray*}
\mu_p(\F)&=& \sum_{B}a_B\mu_p(\F,B)+\sum_{q\in\cl{I}_{p}(\F)}\nu_q(\F)\xi_q(\F)-\nu_p(\F)+\nu_p(\B)-\deg(\B)\\
&=& \sum_{B}a_B\mu_p(\F,B)+\sum_{q\in\cl{I}_{p}(\F)}\nu_q(\F)\xi_q(\F)-\xi_p(\F)-\deg(\B)+1\\
&=& \sum_{B}a_B\mu_p(\F,B)+\chi_p(\F)-\deg(\B)+1.
\end{eqnarray*}
\end{proof}

\par Let $f(x_{1},\ldots, x_{n})\in \mathbb C[x_{1},\ldots, x_{n}]$ be a polynomial, where  the origin is  an isolated singular point of the hypersurface $f^{-1}(0)$. 
The notion of the Milnor number $\mu(f)$ was introduced  in \cite[Section 7]{Milnor} as the degree of the mapping $z\to \frac{\nabla(f)}{\vert\vert \nabla(f)\vert\vert }$, where $\nabla$ denotes the gradient function. In particular, for complex plane curves, Milnor proved, using topological tools, the purely algebraic equality $\mu(f)=2\delta(f)-r(f)+1$, where $\delta(f)$ is the {\em number of double points} and $r(f)$ is the number of irreducible factors of $f$ (see \cite[Theorem 10.5]{Milnor}). The reader can find further formulae for the Milnor number of a plane curve in  \cite[Section 6.5]{Wall}). In particular in \cite[Theorem 6.5.1]{Wall}) it was established the relationship between the Milnor number of a reduced plane curve and the Milnor numbers of its irreducible components. The ingredients of the proof of Wall are Milnor fibrations and the Euler characteristic. We give another proof of this relationship, using foliations:
\begin{proposition}\label{Milnor formula}
Let $C:f(x,y)=0$ be a germ of reduced singular curve at $(\C^2,p)$. Assume that $C=\cup_{j=1}^{\ell}C_{j}$ is the decomposition of $C$ in irreducible components $C_{j}: f_j(x,y)=0$, where $f(x,y)=f_1(x,y)\cdots f_{\ell}(x,y)$. Then 
\[\mu_p(C)+\ell-1=\sum_{j=1}^{\ell}\mu_p(C_{j})+2\sum_{1\leq i<j\leq \ell}i_p(C_i,C_j).\]
\end{proposition}
\begin{proof}
Applying Proposition \ref{Milnor and chi} to the foliation defined by $\omega=df$ and to the balanced divisor of separatrices $C=\sum_{j=1}^{\ell}C_j$ we have 
\begin{equation}\label{eq_milnor}
\mu_p(C)+\ell-1=\sum_{j=1}^{\ell}\mu_p(df,C_j).
\end{equation}
It follows from Proposition \ref{cano_pro} that
\[\mu_p(df,C_j)=i_p(\mathcal{P}^{df},C_j)-\nu_p(C_j)+1,\,\,\,\,\,\,\hbox{\rm for} \,\,j=1,\ldots,\ell.\]
Using properties on intersection numbers, we have
\[i_p(\mathcal{P}^{df},C_j)=i_p(\mathcal{P}^{df_j},C_j)+\sum_{i\neq j}i_p(C_i,C_j).\] 
From Teissier's Proposition \cite[Chapter II, Proposition 1.2]{Teissier}, we get 
\[i_p(\mathcal{P}^{df_j},C_j)=\mu_p(C_j)+\nu_p(C_j)-1.\] Thus 
\begin{equation}\label{eq_polar}
\mu_p(df,C_j)=\mu_p(C_j)+\sum_{i\neq j}i_p(C_i,C_j).
\end{equation}
The proof ends, by substituting (\ref{eq_polar}) in (\ref{eq_milnor}). 
\end{proof}

\begin{maintheorem}\label{Th:Delta-chi}
Let $\F$ be a singular foliation at $(\mathbb C^2,p)$ and let $\B=\sum_{B}a_BB$ be a balanced divisor of separatrices for $\F$. Then
\[
\Delta_p(\F,\B)=\mu_p(\F)-\sum_{B}a_B\mu_p(d{F}_B,B)+ \deg(\B)-1-\chi_p(\F),
\]
where ${F}_B$ is a balanced divisor of separatrices for $\F$ adapted to $B$. Hence, if $\F$ is a foliation of second type, then 
\[
\Delta_p(\F,\B)=\mu_p(\F)-\sum_{B}a_B\mu_p(d{F}_B,B)+ \deg(\B)-1.
\]
\end{maintheorem}
\begin{proof}
By \eqref{END:1} and \eqref{EN:1}  
\begin{eqnarray*}
\Delta_p(\F,{\B})&=&\sum_{B}a_B \Delta_p(\F,B)=\sum_{B}a_B \left( i_p(\mathcal{P}^{\F},B)-i_p(\mathcal{P}^{d{F}_{B}},B)\right)\\
&= &   i_p(\mathcal{P}^{\F},\B)-\sum_{B}a_{B}i_p(\mathcal{P}^{d{F}_{B}},B).
\end{eqnarray*}
Hence, after Lemma \ref{index_2} and Proposition \ref{cano_pro} we have
\begin{eqnarray*}
\Delta_p(\F,{\B})&=&\mu_p(\F)+\nu_p(\F)-\sum_{q\in\cl{I}_{p}(\F)}\nu_q(\F)\xi_q(\F)-\sum_{B}a_{B}\left(\mu_p(d{F}_{B},B)+\nu_p(B)-1 \right)\\
&=&\mu_p(\F)+\nu_p(\F)-\sum_{q\in\cl{I}_{p}(\F)}\nu_q(\F)\xi_q(\F)-\left(\sum_{B}a_{B}\mu_p(d{F}_{B},B)\right) -\nu_p(\B)+\deg(\B).
\end{eqnarray*}

\noindent We finish the proof after Proposition \ref{prop:Equa-Ba} and the definition of the $\chi$-number of $\F$.

\noindent On the other hand if $\F$ is a foliation of second type then $\chi_p(\F)=0$ and the second part of the theorem follows.
\end{proof}

\noindent From Theorem \ref{Th:Delta-chi} and Proposition \ref{Milnor and chi} we get:

\begin{corollary}\label{cor:9}
Let $\F$ be a singular foliation at $(\mathbb C^2,p)$ and let $\B=\sum_{B}a_BB$ be a balanced divisor of separatrices for $\F$. Then

\[
\Delta_p(\F,\B)=\sum_{B}a_B(\mu_p(\F, B)-\mu_p(d{F}_B,B)),
\]
where  ${F}_B$ is a balanced divisor of separatrices for $\F$ adapted to $B$.
\end{corollary}
\par Corollary  \ref{cor:9} restricted to non-dicritical singular foliations provides us a new characterization of non-dicritical generalized curve foliations.

\begin{corollary}
\label{char CG}
Let $\F$ be a singular foliation at $(\mathbb C^2,p)$. Assume that $\F$ is non-dicritical and $C=\cup_{j=1}^{\ell}C_j$ is the total union of separatrices of $\F$. Then
$\F$ is a generalized curve foliation if and only if 
\[\mu_p(\F,C_j)=\mu_p(df,C_j),\,\,\,\,\,\,\,\,\,\,\hbox{\rm for all}\,\,\,\,j=1,\ldots,\ell,\]
where $f(x,y)=0$ is a reduced equation of $C$ at $p$.
\end{corollary}
\begin{proof}
According to \cite[Theorem A]{Genzmer-Mol}, $\F$ is a generalized curve foliation at  $(\mathbb C^2,p)$ if and only if $\Delta_p(\F,C)=0$, where $C$ is the total union of separatrices of $\F$. It follows from Corollary \ref{cor:9} that 
$
\Delta_p(\F,C)=\sum_{j=1}^{\ell}(\mu_p(\F, C_{j})-\mu_p(df,C_{j})),
$
where $f(x,y)=0$ is a  reduced equation of $C$ at $p$.
Thus, by Remark \ref{nonnegative}, $\Delta_p(\F,C)=0$ if and only if $\mu_p(\F,C_j)=\mu_p(df,C_j)$ for all $j=1,\ldots,\ell.$
\end{proof}

\begin{remark}
Observe that if, in Corollary \ref{char CG}, the curve $C: f(x,y)=0$ is irreducible then we rediscover the classic characterization of generalized curve foliations, that is, $\mu_p(\F)=\mu_p(df)=\mu_p(C)$.
\end{remark}

\section{The G\'omez-Mont-Seade-Verjovsky index}
\label{sect:GSV}

Let $\F: \omega=0$ be a singular foliation at $(\mathbb C^2,p)$. Let $C:f(x,y)=0$ be an $\F$-invariant curve, where $f(x,y)\in \C[[x,y]]$ is reduced. Then, as in the convergent case, there are $g,h\in \C[[x,y]]$ (depending on $f$ and $\omega$), with $f$ and $g$ and $f$ and $h$ relatively prime and a $1$-form $\eta$ (see \cite[Lemma 1.1 and its proof]{suwa1995}) such that 
\begin{equation}\label{eq:suwa}
g\omega=hdf+f\eta.
\end{equation}

The {\em G\'omez-Mont-Seade-Verjovsky index} of the foliation $\F$  at $(\C^2,p)$ (GSV-index) with respect to an analytic $\F$-invariant  curve $C$ is
\begin{equation}
\label{eq:GSV}
GSV_p(\F,C)=\frac{1}{2\pi i}\int_{\partial C}\frac{g}{h}d\left(\frac{h}{g}\right),
\end{equation}
where $g,h\in \C\{x,y\}$ are from \eqref{eq:suwa}. This index was introduced in \cite{GSV} but here we follow the presentation of \cite{index}.  If $C$ is irreducible then equality \eqref{eq:GSV} becomes

 \begin{equation}
 \label{eq: irred-GSV}
 GSV_p(\F,C)=\ord_{t}\left(\frac{h}{g}\circ\gamma\right)(t)=i_p(f,h)-i_p(f,g),
\end{equation}
where $\gamma(t)$ is a Puiseux parametrization of $C$. The same formula appears in \cite[Corollary 5.2]{suwa_gsv}. An interesting survey on indices and residues is \cite{correa}.   By \cite[page 532]{index}, we get the adjunction formula
\begin{equation}
\label{adjunction}
GSV_p(\F,C_1\cup C_2)=GSV_p(\F,C_1)+GSV_p(\F,C_2)-2i_p(C_1,C_2),
\end{equation}
 for any two analytic $\F$-invariant  curves,  $C_1$ and $C_2$,  without common irreducible components.

The equality  \eqref{eq: irred-GSV} allows us to extend the definition of the GSV-index to a purely formal (non-analytic) irreducible $\F$-invariant  curve; and the equality \eqref{adjunction} allows us to extend the definition of the GSV-index to $\F$-invariant curves containing purely formal branches.

The following lemma generalizes the equality \eqref{eq: irred-GSV} to any reduced $\F$-invariant curve  (containing perhaps purely formal branches):

\begin{lemma}
\label{eq:GSVgeneral}
Let $C:f(x,y)=0$ be any reduced invariant curve of a singular foliation $\F$ at $(\C^2,p)$. Then 
\[
GSV_p(\F,C)=i_p(f,h)-i_p(f,g),
\]
where $g,h\in \C[[x,y]]$ are from \eqref{eq:suwa}.
\end{lemma}

\begin{proof}
Suppose, without lost of generality, that $f(x,y)=f_1(x,y)f_2(x,y),$ where $f_1,f_2 \in \mathbb C[[x,y]]$ are irreducible and put $C_i: f_i(x,y)=0$ for $1\leq i\leq 2$. By \eqref{eq:suwa} we get
\[
g\omega=hf_2df_1+f_1(hdf_2+f_2\eta),
\]

\noindent for some $g,h\in \mathbb C[[x,y]]$ relative prime with $f$ and a $1$-form $\eta$. Hence, if $\gamma_1(t)$ is a Puiseux parametrization of $C_1$, then after \eqref{eq: irred-GSV} we have

\begin{eqnarray*}
GSV_p(\F, C_1)&=&\ord_{t}\left(\frac{hf_2}{g}\circ \gamma_1 \right)(t)=\ord_{t}\left(\frac{h}{g}\circ \gamma_1 \right)(t)+\,\ord_{t}\left({f_2}\circ \gamma_1\right)(t)\\
&=&\ord_{t}\left(\frac{h}{g}\circ \gamma_1\right)(t)+i_p(C_1,C_2).
\end{eqnarray*}
Similarly, if  $\gamma_2(t)$ denotes a Puiseux parametrization of $C_2$ then  we have
\[
GSV_p(\F, C_2)=\ord_{t}\left(\frac{h}{g}\circ \gamma_2\right)(t)+i_p(C_1,C_2).
\]
The proof follows after  equality \eqref{adjunction} and properties of the intersection number.
\end{proof}
\par In this section, we will use the following result due to Genzmer-Mol \cite[Theorem B]{Genzmer} that establishes a relationship between the GSV-index and the polar excess number of a foliation with respect to a set of separatrices. 
\begin{theorem}\label{mol}
Let $\F$ be a  singular foliation at $(\C^2,p)$. Let $C$ be a reduced curve of separatrices and $\B=\B_0-\B_{\infty}$ be a balanced divisor of separatrices for $\F$ adapted to $C$. Then 
$$GSV_p(\F,C)=\Delta_p(\F,C)+i_p(C,\B_{0}\setminus C)-i_p(C,\B_{\infty}).$$
\end{theorem}
\par We note that the above theorem implies that if $\B$ is an effective balanced divisor of separatrices for $\F$, then $GSV_p(\F,\B)=\Delta_p(\F,\B)$.
\par On the other hand, as a consequence of Proposition \ref{prop-B} and Theorem \ref{mol} we get

\begin{corollary}
\label{GSV-zero divisor}
Let $\F$ be a singular foliation at $(\C^2,p)$. Let $\B=\B_0-\B_{\infty}$ be a reduced balanced divisor of separatrices for $\F$. Then 
\[GSV_p(\F,\B_0)=i_{p}(\mathcal P^{\F},\B_{0})-\mu_p(\B_{0})-\nu_p(\B_{0})+1.\]
\end{corollary}
\begin{proof}
By applying Theorem \ref{mol} to $C=\B_0$, we have
\begin{equation}\label{eq_3}
GSV_p(\F,\B_0)=\Delta_p(\F,\B_0)-i_p(\B_0,\B_{\infty}),
\end{equation}
and it follows from Proposition \ref{prop-B} that
\begin{equation}\label{eq_4}
\Delta_p(\F,\B_{0})=i_{p}(\mathcal P^{\F},\B_{0})+i_{p}(\B_{0},\B_{\infty})-\mu_p(\B_{0})-\nu_p(\B_{0})+1.
\end{equation}
The proof ends by substituting (\ref{eq_4}) in (\ref{eq_3}).
\end{proof}

The following proposition holds for an arbitrary foliation and any subset of separatrices and is not restricted only to convergent separatrices as in \cite[Proposition 4]{Cano}. The proof is similar, and is written for the reader's understanding.

\begin{proposition}\label{gsv=delta}
Let $C:f(x,y)=0$ be any reduced invariant curve of a singular foliation $\F$ at $(\C^2,p)$. Then 
\[GSV_p(\F,C)=i_p(\mathcal P^{\mathcal F},C)-i_p(\mathcal P^{df},C).\]
\end{proposition}
\begin{proof}
Let $\omega=P(x,y)dx+Q(x,y)dy$  be a 1-form inducing $\F$.
By equality \eqref{eq:suwa} we get $g\omega=hdf+f\eta,$
where $\eta$ is a formal 1-form and $g, h\in\C[[x,y]]$ 
with $g$ and $f$  relatively prime and $h$ and $f$ also relatively prime. Let $(a:b)\in\mathbb{P}^1$ such that the polar curves $aP(x,y)+bQ(x,y)=0$ and $a\partial_x{f}+b\partial_y{f}=0$ of $\F$ and $df$ respectively, are generic. We have
$g\cdot (aP+bQ)=h\cdot (a\partial_x{f}+b\partial_y{f})+fk,$
for some $k\in\C[[x,y]]$.
Then $i_p(f,g\cdot (aP+bQ))=i_p(f,h\cdot (a\partial_x{f}+b\partial_y{f})+fk)=i_p(f, h\cdot (a\partial_x{f}+b\partial_y{f}))$. So $i_p(f,g\cdot (aP+bQ))=i_p(f,h)+i_p(f,a\partial_x{f}+b\partial_y{f})$.

\noindent On the other hand  $i_p(f,g\cdot (aP+bQ))=i_p(f,g)+i_p(f,aP+bQ)$, hence
\begin{equation}
\label{eq:restando}
i_p(f,g)+i_p(f,aP+bQ)=i_p(f,h)+i_p(f,a\partial_x{f}+b\partial_y{f}).
\end{equation}
 
 \noindent We finish the proof using Lemma \ref{eq:GSVgeneral} and equality \eqref{eq:restando}.

\end{proof}

\begin{remark}
If $\F$ is a non-dicritical singular foliation at $(\C^2,p)$, where $C$ is the total union of separatrices of $\F$ then $i_p(\mathcal P^{\mathcal F},f)-i_p(\mathcal P^{df},f)=\Delta_p(\F,C)$, but in general this two values are  different as the following example shows: consider the foliation $\F$ 
 defined by $\omega=2xdy-3ydx$ and  the curve $C:$ $y^2-x^3=0$. 
Note that $\F$ admits the meromorphic first integral $y^2/x^3$ and so that $C$ is $\F$-invariant. We get
$i_0(\mathcal P^{\mathcal F},C)-i_0(\mathcal P^{df},C)=-1,$
and $\Delta_0(\F,C)=0$, since $\F$ is a generalized curve foliation. 
\end{remark}

We obtain the following corollary.
\begin{corollary}\label{gsv=milnor}
Let $\F$ be a singular foliation at $(\mathbb C^2,p)$. Assume that $\F$ is non-dicritical and $C$ is the total union of separatrices of $\F$. Then
\[GSV_p(\F,C)=\mu_p(\F)-\mu_p(C)-\chi_p(\F).\]
\end{corollary}
\begin{proof}
By Proposition \ref{gsv=delta} we have $GSV_p(\F,C)=i_p(\mathcal P^{\mathcal F},C)-i_p(\mathcal P^{df},C)$ and applying Lemma \ref{index_2} to $\F$ and $C$, we get
$i_p(\mathcal P^{\mathcal F},C)=\mu_p(\F)+\nu_p(\F)-\sum_{q\in\mathcal{I}_p(\F)}\nu_q(\F)\xi_q(\F)$.
Teissier's Proposition \cite[Chapter II, Proposition 1.2]{Teissier} implies that
\[i_p(\mathcal{P}^{df},C)=\mu_p(C)+\nu_p(C)-1.\] Thus 
\begin{eqnarray*}
GSV_p(\F,C)&=& i_p(\mathcal P^{\mathcal F},C)-i_p(\mathcal P^{df},C)\\
&=& \mu_p(\F)+\nu_p(\F)-\sum_{q\in\mathcal{I}_p(\F)}\nu_q(\F)\xi_q(\F) - (\mu_p(C)+\nu_p(C)-1)\\
&=& \mu_p(\F)-\mu_p(C)+\underbrace{(\nu_p(\F)-\nu_p(C)+1)}_{\xi_p(\F)}-\sum_{q\in\mathcal{I}_p(\F)}\nu_q(\F)\xi_q(\F)\\
&=& \mu_p(\F)-\mu_p(C) - \left(\underbrace{\sum_{q\in\mathcal{I}_p(\F)}\nu_q(\F)\xi_q(\F)-\xi_p(\F)}_{\chi_p(\F)}\right)\\
&=& \mu_p(\F)-\mu_p(C)-\chi_p(\F).
\end{eqnarray*}
\end{proof}

To finish this section we state a relationship between the GSV-index and the multiplicity of $\F$ along a fixed separatrix. 
\begin{proposition}
\label{prop:gsv and multiplicity}
Let $\F$ be a singular foliation at $(\C^2,p)$ and $B:f(x,y)=0$ be a separatrix of $\F$. 
Then 
\begin{equation}
\mu_p(\F,B)=
\begin{cases}
GSV_p(\F,B)+\ord_t\partial_y f(\gamma(t))-\ord_t x(t)+1 & \text{if $x(t)\neq 0$};
\medskip \\
GSV_p(\F,B)+\ord_t\partial_x f(\gamma(t))-\ord_t y(t)+1 & \text{if $y(t)\neq 0$},
\end{cases} 
\end{equation}
where $\gamma(t)=(x(t),y(t))$ is a Puiseux parametrization of $B$. In particular, if $B$ is a non-singular separatrix, then $\mu_p(\F,B)=GSV_p(\F,B)$.
\end{proposition} 
\begin{proof}
Let $\omega=P(x,y)dx+Q(x,y)dy$  be a 1-form inducing $\F$ and $f(x,y)=0$ be a reduced equation of $B$. By equality \eqref{eq:suwa} we get
\[
g\omega=hdf+f\eta,
\]
where $\eta$ is a formal 1-form and $g, h\in\C[[x,y]]$, 
where $g$ and $f$ are relatively prime and $h$ and $f$ are relatively prime. From equality (\ref{mult_2}), we have that the unique vector field $\theta(t)$ such that 
$\gamma_{*}\theta(t)=v(\gamma(t))$, where $v=-Q(x,y)\frac{\partial}{\partial{x}}+P(x,y)\frac{\partial}{\partial{y}}$, is given by 
\begin{equation}
\theta(t)=
\begin{cases}
\dfrac{-\left(\frac{h}{g}\right)(\gamma(t))\partial_y f(\gamma(t))}{x'(t)} & \text{if $x(t)\neq 0$};
\medskip \\

\dfrac{\left(\frac{h}{g}\right)(\gamma(t))\partial_x f(\gamma(t))}{y'(t)} & \text{if $y(t)\neq 0$}.
\end{cases} 
\end{equation}
Therefore, the proof follows taking orders. 
\end{proof}

\section{Tjurina  number}
\label{sect:Tjurina}
Let $\F$ be a singular foliation at $(\C^2,p)$ defined by the 1-form $\omega=P(x,y)dx+Q(x,y)dy$
 and $C:f(x,y)=0$ be a $\F$-invariant reduced curve.  The {\it Tjurina number of} $\F$  {\it with respect to} $C$ is
\[\tau_p(\F,C)=\dim_{\mathbb  C} \mathbb  C[[x,y]]/(f,P,Q).\]
%\noindent where $(\cdot, \cdot, \cdot)$ denotes the ideal generated by three elements in $\mathbb C[[x,y]]$. If ${\mathcal B}=\sum_{B}a_{B}B$ is a divisor of separatrices for $\F$ then we define the {\it Tjurina number of} $\F$  {\it with respect to} ${\mathcal B}$ as 
%\[
%\tau_p(\F,\mathcal B)=\sum_{B}a_{B}\tau_p(\F,B).
%\]

\par The {\it Tjurina number} of any germ of reduced curve $C:f(x,y)=0$, with $f(x,y)\in \mathbb  C[[x,y]]$ is by definition 
\[\
\tau_p(C)=\dim_{\mathbb C} \mathbb  C[[x,y]]/(f,\partial_{x} f,\partial_{y} f).\]

\par In this section we will study the Tjurina number of a foliation with respect to a balanced divisor of separatrices. First of all we present a lemma on Commutative Algebra which we need in the sequel and we did not find it in the literature:

\begin{lemma}
\label{ca}
Let $f,g,p,q\in \mathbb C[[x,y]],$ where $f$ and $g$ are relatively prime. Then 
\[
\dim_{\mathbb C}\mathbb C[[x,y]]/(f,gp,gq)=\dim_{\mathbb C}\mathbb C[[x,y]]/(f,p,q)+\dim_{\mathbb C}\mathbb C[[x,y]]/(f,g).
\]
\end{lemma}
\begin{proof}
Observe that $\dim_{\mathbb C}\mathbb C[[x,y]]/(f,r_{1},\ldots,r_{n})=\dim_{\mathbb C}\mathcal O/(r'_{1},\ldots,r'_{n}),$ where $\mathcal O=\mathbb C[[x,y]]/(f)$ and $r'_{i}=r_{i}+(f)$ for any $i\in\{1, \ldots,n\}$ and
any $r_{i}\in \mathbb C[[x,y]]$. We finish the proof using the following exact sequence:
\[
0\longrightarrow \mathcal O/(p',q') \stackrel{\sigma}{\longrightarrow} \mathcal O/(g'p',g'q')\stackrel{\delta}{\longrightarrow} \mathcal O/(g')\longrightarrow 0,
\]
where $\sigma(z'+(p',q'))=g'z'+(g'p',g'q')$ and $\delta(z'+(g'p',g'q'))=z'+(g')$, for any $z'\in \mathcal O$.
\end{proof}
\par The following proposition has been proved by X. G\'omez-Mont \cite[Theorem 1]{Gomez} for a foliation with a set of convergent separatrices. We show that the same result holds in the formal context for effective reduced balanced divisor of separatrices. 

\begin{proposition}\label{obs_1}
Let $\F$ be a singular foliation at $(\C^2,p)$ and $C$ be  a reduced curve of separatrices of $\F$. Then
\[\tau_p(\F,C)-\tau_p(C)=GSV_p(\F,C).\]
\end{proposition}
\begin{proof}
Let $\omega=P(x,y)dx+Q(x,y)dy$  be a 1-form inducing $\F$ and $f(x,y)=0$ be the reduced equation of $C$.
By equality \eqref{eq:suwa} we get
$g\omega=hdf+f\eta,$
where $\eta$ is a formal 1-form and $g, h\in\C[[x,y]]$ 
with $g$ and $f$  relatively prime and $h$ and $f$ also relatively prime.

Hence
$gPdx+gQdy=(h\partial_x{f}+f\eta_x)dx+(h\partial_y{f}+f\eta_y)dy,$
where $\eta=\eta_xdx+\eta_y dy$. We get 
\begin{equation}
\label{eq:1}gP=h\partial_x{f}+f\eta_x,\,\,\,\,\,\text{and}\,\,\,\,\, gQ=h\partial_y{f}+f\eta_y.
\end{equation}
\noindent After equalities \eqref{eq:1}, properties of the intersection number and Lemma \ref{ca}, we have 
\begin{eqnarray*}
\dim_{\C}\C[[x,y]]/(f,gP,gQ)&=&\dim_{\C}\C[[x,y]]/(f,h\partial_x{f}+f\eta_x,h\partial_y{f}+f\eta_y)\\
&=&\dim_{\C}\C[[x,y]]/(f,h\partial_x{f},h\partial_y{f})\\
&=&\dim_{\C}\C[[x,y]]/(f,h)+\dim_{\C}\C[[x,y]]/(f,\partial_x{f},\partial_y{f})\\
&=&\dim_{\C}\C[[x,y]]/(f,h)+\tau_p(C). \label{eq:2}
\end{eqnarray*}
Again, by Lemma \ref{ca} we get 
\[\dim_{\C}\C[[x,y]]/(f,P,Q)+\dim_{\C}\C[[x,y]]/(f,g)=\dim_{\C}\C[[x,y]]/(f,h)+\tau_p(C).\]
Hence $\tau_p(\F,C)-\tau_p(C)=\dim_{\C}\C[[x,y]]/(f,h)-\dim_{\C}\C[[x,y]]/(f,g)=i_p(f,h)-i_p(f,g)$.  The proof follows from Lemma \ref{eq:GSVgeneral}.
\end{proof}

\begin{example}
Let $\F$ be the foliation defined by the formal normal form of a saddle-node (see \eqref{saddle-node-formal}) $\omega=x^{k+1}dy-y(1+\lambda x^{k})dx,\,\,\,\,\,\,k\geq 1,\,\,\,\,\,\lambda\in\C.$
The total union of separatrices of $\F$ is $C=C_1\cup C_2$, where $C_1: x=0$ (strong separatrix) and $C_2: y=0$ (weak separatrix). An equality  \eqref{eq:suwa} for 
$C_{1}$ is given for $g=1$, $h=-y(1+\lambda x^{k})$ and $\eta=x^{k}dy$, hence by Lemma \ref{eq:GSVgeneral} we get $GSV_0(\F,C_1)=i_{0}(x,h)-i_{0}(x,g)=1$. Similarly, an equality \eqref{eq:suwa} for 
$C_{2}$ is given for $g=1$, $h=x^{k+1}$ and $\eta=-(1+\lambda x^{k})dx$, thus $GSV_0(\F,C_2)=i_{0}(y,h)-i_{0}(y,g)=k+1$. Therefore, one finds
 \[GSV_0(\F,C)=GSV_0(\F,C_1)+GSV_0(\F,C_2)-2i_0(C_1,C_2)=1+(k+1)-2=k.\]
On the other hand, we get $\tau_0(\F,C)-\tau_0(C)=(k+1)-1=k=GSV_0(\F,C).$

\end{example}

\begin{maincorollary}
Let $\F$ be a singular foliation at $(\mathbb C^2,p)$. Assume that $\F$ is non-dicritical and $C$ is the total union of separatrices of $\F$. Then
 \[
 \mu_{p}(\F)-\tau_{p}(\F,C)=\mu_{p}(C)-\tau_{p}(C)+\chi_{p}(\F).
 \]
 Moreover, if $\F$ is of second type then $\mu_{p}(\F)-\tau_{p}(\F,C)=\mu_{p}(C)-\tau_{p}(C)$.
\end{maincorollary}

\begin{proof}
It is a consequence of Proposition \ref{obs_1} and Corollary \ref{gsv=milnor}.
\end{proof}

\par Now, we characterize generalized curve foliations at $(\C^2,p)$ in terms of the Tjurina numbers. 
\begin{corollary}
Let $\F$ be a singular foliation at $(\C^2,p)$ and $\B=\B_0-\B_{\infty}$ be a reduced balanced divisor of separatrices for $\F$. 
Then $\F$ is a generalized curve if and only if 
$\tau_p(\B_0)-\tau_p(\F,\B_0)=i_p(\B_0,\B_{\infty}).$
\end{corollary}
\begin{proof}
It follows from \cite[Theorem A]{Genzmer-Mol} that $\F$ is a generalized curve foliation  if and only if $\Delta_p(\F,\B_0)=0$. Applying Theorem \ref{mol} to $C=\B_0$, we get
$GSV_p(\F,\B_0)=\Delta_p(\F,\B_0)-i_p(\B_0,\B_{\infty})$. Hence $\F$ is a generalized curve foliation  if and only if $GSV_p(\F,\B_0)=-i_p(\B_0,\B_{\infty})$. The proof ends, by applying Proposition \ref{obs_1} to $C=\B_0$.
\end{proof}

If ${\mathcal B}=\sum_{B}a_{B}B$ is a divisor of separatrices for $\F$ then we put 
\[
T_p(\F,\mathcal B)=\sum_{B}a_{B}\tau_p(\F,B).
\]

\par The following theorem gives a relationship between the Milnor and Tjurina numbers and the $\chi$-number, studied in Section \ref{sect:chi number}.

\begin{secondtheorem}\label{excess}
Let $\F$ be a singular foliation at $(\mathbb C^2,p)$ and let $\B=\sum_{B}a_BB$ be a balanced divisor of separatrices for $\F$. Then

\begin{eqnarray*}
\mu_p(\F)- T_p(\F,\B)&=&\sum_{B}a_B[ \mu_p(dF_B,B) -\tau_p(B)]-\deg(\B)+1+\chi_p(\F)\\
& & - \sum_{B} a_B[i_p(B,(F_B)_0\setminus B)-i_p(B,(F_B)_{\infty})], 
\end{eqnarray*}

where ${F}_B$ is a balanced divisor of separatrices for $\F$ adapted to $B$.
\end{secondtheorem}
\begin{proof}
%From definition of the Tjurina number of $\F$ with respect to $\B$, we have
%\[\tau_p(\F,\B)=\sum_{B}  a_B \tau_p(\F,B).\]
%Therefore, 

By Proposition \ref{obs_1} we get 
$T_p(\F,\B)=\sum_{B} a_B (GSV_p(\F,B)+\tau_p(B)).$ Then $T_p(\F,\B)-\sum_{B}a_B \tau_p(B)= \sum_{B} a_B GSV_p(\F,B)$. 
From Theorem \ref{mol}, we have 
\begin{eqnarray*}\label{eq_6}
T_p(\F,\B)-\sum_{B}a_B \tau_p(B)&=& \sum_{B} a_B[\Delta_p(\F,B)+i_p(B,(F_B)_0\setminus B)-i_p(B,(F_{B})_{\infty})]\nonumber\\
&=& \Delta_p(\F,\B)+\sum_{B} a_B[i_p(B,(F_B)_0\setminus B)-i_p(B,(F_{B})_{\infty})],
\end{eqnarray*}
where ${F}_B$ is a balanced divisor of separatrices for $\F$ adapted to $B$. We finish the proof using Theorem \ref{Th:Delta-chi}.
\end{proof}

\par In order to illustrate Theorem \ref{excess} we present a family of dicritical foliations that are not of second type.

\begin{example}
\label{ex:Arturo}
Let $\lambda\in\C$ and $k\geq 3$ integer. Let $\F_k$ be the singular foliation at $(\C^2,0)$ defined by
\[\omega_k=y(2x^{2k-2}+2(\lambda+1)x^2y^{k-2}-y^{k-1})dx+x(y^{k-1}-(\lambda+1)x^2y^{k-2}-x^{2k-2})dy.\]
The foliation $\F_k$ is  dicritical. 
\begin{figure}[H]
\centering
			\includegraphics[width=12.5cm, angle=0]{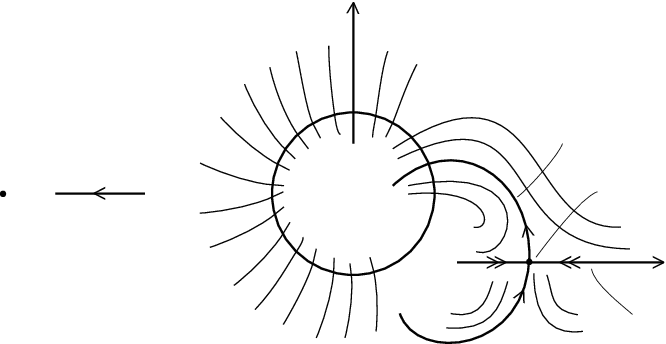}
			\put(-356.5,77){$\bullet$}
			\put(-356,69){\scriptsize {$0$}}
			\put(-160,177){\scriptsize {$y$}}
			\put(-07,32){\scriptsize {$x$}}
			\put(-55,85){\scriptsize {Saddle-node singularity}}
			\put(-40,7){\scriptsize {Strong separatrix}}
			\put(-80,78){$\vector(-1,-1){0.2}$}
			\put(-70,46.5){$\vector(-1,-1){0.2}$}
			\put(-41.9,41){$\vector(-1,1){0.2}$}
			\put(-78,109){\scriptsize {Weak separatrix}}
	\caption{Dicritical foliation $\F_k$}
	\end{figure}
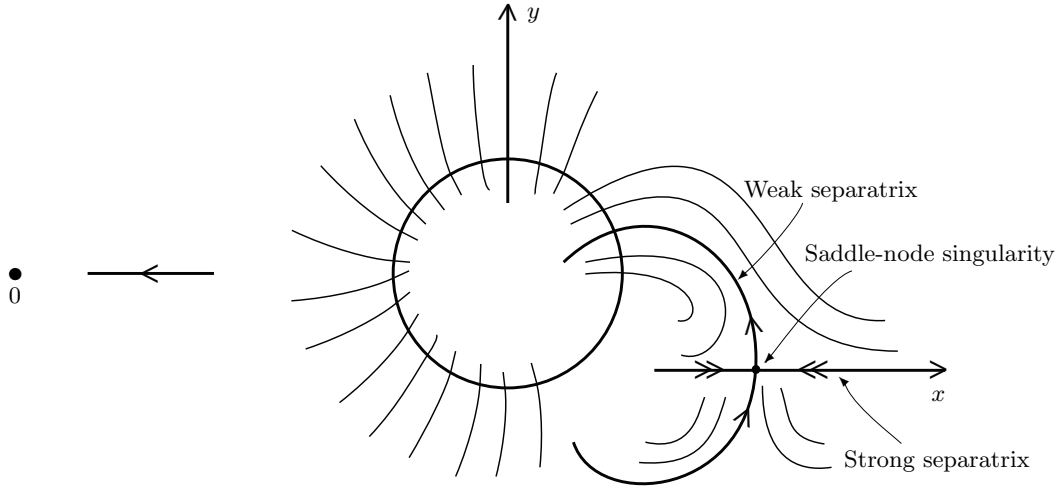
After one blow-up, the foliation has a unique non-reduced singularity $q$.  A further
blow-up applied to $q$ produces a reduction of singularities of the foliation with a dicritical component and a tangent saddle-node with strong separatrix transversal to exceptional divisor. Therefore $\F_k$ is not of second type. Let $B_1:y=0$ and $B_2:x=0$, then  $\B=B_1+B_2$ is an effective balanced divisor of separatrices  for $\F_k$. 
A simple calculation leads to:
\[\nu_0(\F_k)=k,\,\,\,\nu_q(\F_k)=k-1,\,\,\,\,\xi_0(\F_k)=k-1,\,\,\,\xi_q(\F_k)=k-1,\]
thus $\chi_p(\F_k)=2(k-1)^{2}$.
Moreover $\mu_0(\F_k)=(k-2)(2k-2)+5k-4$,
$T_0(\F_k,\B)=3k-1$, $\tau_0(B_1)=\tau_0(B_2)=0$. Since $F(x,y)=xy$ defines a balanced divisor of separatrices for $\F_k$ adapted to $B_1$ and $B_2$, we get 
$\mu_0(dF,B_1)+\mu_0(dF,B_2)=2.$
Hence we have $T_0(\F_k,\B)=3k-1$
%\[\tau_0(\F_k,\B)=3k-1=(k-2)(2k-2)+5k-5-2+2-1+(k-1)-k(k-1)-(k-1)(k-1)+1+1,\]
and Theorem \ref{excess} is verified. 
\end{example}

\begin{corollary}\label{cor:1}
Let $\F$ be a singular foliation at $(\mathbb C^2,p)$ and let $\B=\sum_{B}a_BB$ be a balanced divisor of separatrices for $\F$. If $\F$ is of second type, then

\begin{eqnarray*}
\mu_p(\F)- T_p(\F,\B)&=&\sum_{B}a_B[ \mu_p(dF_B,B) -\tau_p(B)]-\deg(\B)+1\\
& & - \sum_{B} a_B[i_p(B,(F_B)_0\setminus B)-i_p(B,(F_B)_{\infty})], 
\end{eqnarray*}

\noindent where ${F}_B$ is a balanced divisor of separatrices for $\F$ adapted to $B$.

\end{corollary}
\begin{proof} 
By Lemma \ref{prop_chi} item (2),  $\chi_p(\F)=0$ and the proof follows from Theorem \ref{excess}.
\end{proof}

\par The following example shows that, in general, the reciprocal of Corollary \ref{cor:1} is not true.

\begin{example}[{\bf Dulac's foliation}]
The foliation ${\mathcal F}$ defined by the $1$-form
${\omega}=(ny+x^{n})dx-xdy$, with $n\geq 2$, admits a unique separatrix $C:$ $x=0$. Since $\nu_0(\F)=1 \not =0=\nu_0(C)-1$,  so $\F$ is not of second type. 
Moreover $T_0(\F,C)=1$, $\tau_0(C)=0$, $\mu_0(\F)=1$, $\mu_0(dx,C)=0$. Hence, $\F$ verifies the equality of Corollary \ref{cor:1} but it is not a foliation of second type at $0\in\C^2$.
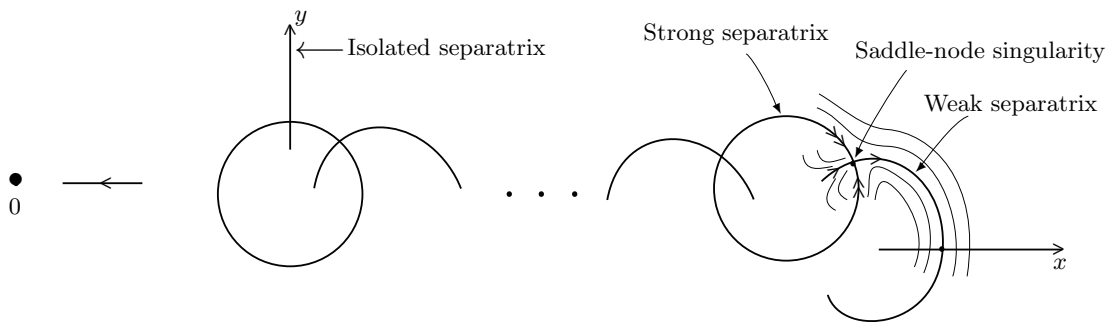
\begin{figure}[H]
\centering
			\includegraphics[width=14cm, angle=0]{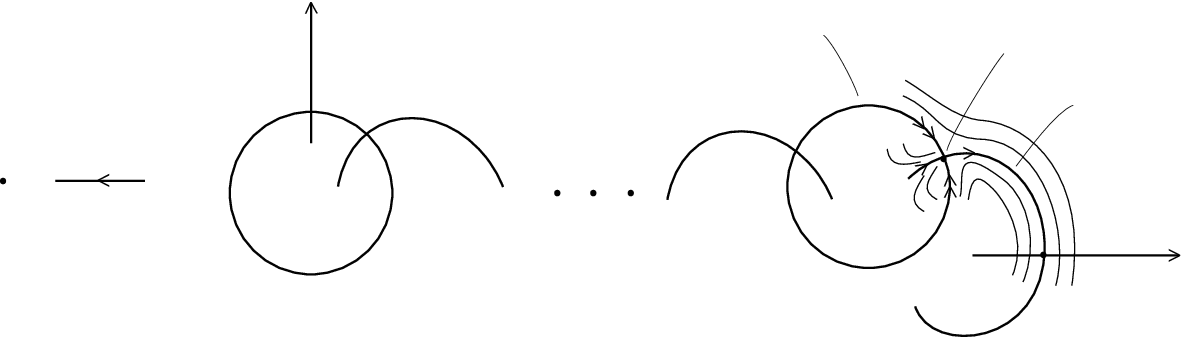}
			\put(-400.5,51){$\bullet$}
			\put(-400,41){\scriptsize {$0$}}
			\put(-292,114){\scriptsize {$y$}}
			\put(-292,100){$\longleftarrow$}
			\put(-272,101){\scriptsize {Isolated separatrix}}
			\put(-5,20){\scriptsize {$x$}}
			\put(-160,107){\scriptsize {Strong separatrix}}
			\put(-80,100){\scriptsize {Saddle-node singularity}}
			\put(-54,80){\scriptsize {Weak separatrix}}
			\put(-108.5,80){$\vector(1,-2){0.5}$}
			\put(-79.5,63.3){$\vector(-1,-2){0.5}$}
			\put(-57,57){$\vector(-1,-1){0.}$}
\caption{Dulac foliation $\F: \omega=(ny+x^{n})dx-xdy$}
\end{figure}
\end{example}

\par Now, we apply Theorem \ref{excess} to non-dicritical singular foliations at $(\C^2,p)$.
\begin{secondcorollary}\label{cor:6}
Let $\F$ be a singular foliation at $(\mathbb C^2,p)$. Assume that $\F$ is non-dicritical and $C=\cup_{j=1}^{\ell}C_j$ is the total union of separatrices of $\F$. Then
\[\mu_p(\F)-T_p(\F,C)=\mu_p(C)-\sum_{j=1}^{\ell}\tau_p(C_{j})+\chi_p(\F)-\sum_{j=1}^{\ell}i_p(C_j,C\setminus C_j).\]
\end{secondcorollary}
\begin{proof}
By taking the effective divisor $\B=(f)$, where $f(x,y)=0$ is an equation of $C$, and applying Proposition \ref{Milnor and chi} to the foliation $df$, we get 
\begin{equation}\label{eq_mi}
\mu_p(C)=\sum_{j=1}^{\ell}\mu_p(df,C_j)-\ell+1.
\end{equation}

The proof follows  from Theorem \ref{excess}.
\end{proof}

\begin{corollary}
Let $\F$ be a singular foliation at $(\mathbb C^2,p)$. Assume that $\F$ is non-dicritical and $C=\cup_{j=1}^{\ell}C_j$ is the total union of separatrices of $\F$. Then
\[
T_{p}(\F,C)-\tau_{p}(\F,C)=\sum_{j=1}^{\ell}\tau(C_{j})-\tau(C)+2\sum_{1\leq i<j\leq \ell}i_p(C_i,C_j).
\]
\end{corollary}

Here is an example to illustrate Corollary \ref{cor:6}.

\begin{example} 
Let $\omega=4xydx+(y-2x^{2})dy$ be a $1$-form defining  a  singular foliation $\F$ at $(\C^2,0)$.
The unique separatrix of $\F$ is the curve $C: y=0$. Since $\nu_0(\F)=1$ and $\nu_0(C)=1$, the foliation  $\F$ is not of second type at the origin (see Proposition \ref{prop:Equa-Ba}).
Moreover, we get $T_0(\F,C)=2$, $\tau_0(C)=0$, $\mu_0(\F)=3$, $\mu_0(C)=0$, and $\chi_0(\F)=1$.
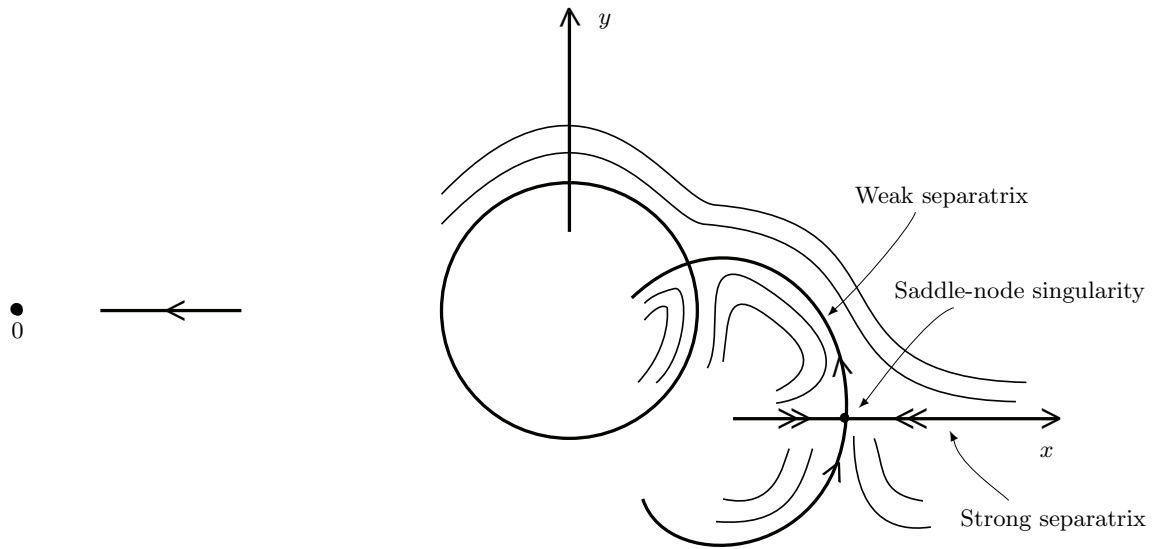
\begin{figure}[H]
\centering
			\includegraphics[width=14cm, angle=0]{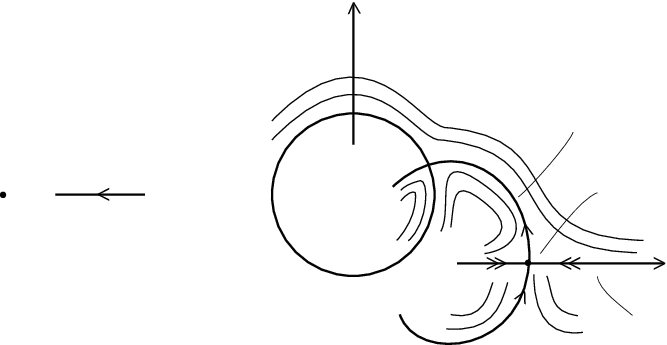}
			\put(-400,87){$\bullet$}
			\put(-399,79){\scriptsize {$0$}}
			\put(-177,198){\scriptsize {$y$}}
			\put(-10,34){\scriptsize {$x$}}
			\put(-40,8){\scriptsize {Strong separatrix}}
			\put(-65,94){\scriptsize {Saddle-node singularity}}
			\put(-80,130){\scriptsize {Weak separatrix}}
			\put(-89.7,88){$\vector(-1,-1){0.5}$}
			\put(-78,53){$\vector(-1,-1){0.5}$}
			\put(-43.6,43){$\vector(-1,2){0.5}$}
\caption{Foliation ${\F}: \omega=4xydx+(y-2x^{2})dy$}
	\end{figure}
\end{example}

\subsection{Milnor and Tjurina numbers and some residue-type indices}
\label{sect: residue}
\par We finish this section stating numerical relationships between some classic indices, such as the Baum-Bott, Camacho-Sad and variational indices,  of a singular foliation at $(\C^2,p)$ and the  Milnor and Tjurina numbers. 

Let $\F$ a singular foliation defined by a $1$-form $\omega$ as in \eqref{vectorfield}. Let $J(x,y)$ be the Jacobian matrix of $(Q(x,y),-P(x,y))$. The {\em Baum-Bott index} (see \cite{Baum}) of $\F$ at  $p$  is 
\[
\bb_{p}(\F)={\rm Res}_{p}\left\{\frac{({\rm tr}J)^{2}}{-P\cdot Q}dx\wedge dy\right\},
\]
where ${\rm tr}J$ denotes the trace of $J$.
\par The {\em Camacho-Sad index} of $\F$ (CS index) with respect to an analytic $\F$-invariant  curve $C$ is
\begin{equation}
\label{eq:CS}
\cs_{p}(\F,C)=-\frac{1}{2\pi i}\int_{\partial C}\frac{1}{h}\eta, 
\end{equation}
where $g,h$ are from \eqref{eq:suwa}. The Camacho-Sad index was introduced by these authors in \cite{CS} for a non-singular $\F$-invariant curve $C$. Later, Lins Neto in \cite{LN} and Suwa in \cite{suwa1995} generalize this index to singular  $\F$-invariant curves.
If $C$ is irreducible then equality \eqref{eq:CS} becomes
 \begin{equation}
 \label{eq: irred-CS}
 \cs_p(\F,C)=-{\rm Res}_{t=0}\left(\gamma^{*}\frac{1}{h}\eta\right),
\end{equation}
where $\gamma(t)$ is a Puiseux parametrization of $C$. By \cite[page 38]{Brunella-libro} (see also \cite{suwa}) we get the adjunction formula
\begin{equation}
\label{adjunction CS}
\cs_{p}(\F,C_1\cup C_2)=\cs_{p}(\F,C_1)+\cs_{p}(\F,C_2)+2i_p(C_1,C_2),
\end{equation}
 for any two analytic $\F$-invariant  curves,  $C_1$ and $C_2$,  without common irreducible components. The equality  \eqref{eq: irred-CS} allows us to extend the definition of the CS index to a purely formal (non-analytic) irreducible $\F$-invariant  curve; and the equality \eqref{adjunction CS} allows us to extend the definition of the CS index to $\F$-invariant curves containing purely formal branches.
\par In a neighborhood of a non-singular point of the foliation $\F$, there is a $1$-form $\alpha$ such that $d\omega=\alpha \wedge \omega$. If $\alpha'$ is other such $1$-form, then $\alpha$ and $\alpha'$  coincide over every leaf of $\F$. Hence, in a neighborhood of $0$ (away $0$) there exists a holomorphic multi-valued $1$-form $\alpha$ such that $d\omega=\alpha \wedge \omega$ and that its restriction to each leaf of $\F$ is single-valued. We say that $\alpha$ is an \textit{exponent form} for $\omega$. The \textit{variational index} or \textit{variation} of $\F$ relative to $C$ at $p$ is 
\[
\var_p(\F,C)={\rm Res}_{t=0}\left(\alpha_{\big|_{C}}\right)= \frac{1}{2\pi i}\int_{\partial C}\alpha.\]
The variational index was introduced in \cite{variation}. It is additive:
\[
\var_p(\F,C_1\cup C_2)=\var_p(\F,C_1)+\var_p(\F,C_2),
\]
where $C_1$ and $C_2$ are $\F$-invariant curves without common factors. For any divisor $\B=\sum_B a_B B$ of separatrices for $\F$ we put
\[
\var_p(\F,\B)=\sum_B a_B \var_p(\F,B).
\]
For any analytic $\F$-invariant curve $C$ we have, after \cite[Proposition 5]{index},
\begin{equation}\label{var_1}
\var_p(\F,C)=\cs_p(\F,C)+GSV_p(\F,C).
\end{equation}

%
%we use the following theorem given in \cite[Theorem 5.2]{Fernandez}. \begin{theorem}\label{baum-bott}
%Let $\F$ be a singular foliation at $(\mathbb C^2,p)$ and $\B$ be a balanced divisor of separatrices for $\F$. Then
%\[\bb_p(\F)=\var_p(\F,\B)+\Delta_p(\F,\B)+\sum_{q\in\mathcal{I}_{p}(\F)}\xi^{2}_q(\F),\]
%where the summation runs over all infinitely near points of $\F$ at $p$.
%\end{theorem}
%\par We apply Theorem \ref{Th:Delta-chi} to $\Delta_p(\F,\B)$ to obtain the following corollary.
\begin{proposition}\label{cor:12}
Let $\F$ be a singular foliation at $(\mathbb C^2,p)$ and let $\B=\sum_{B}a_BB$ be a balanced divisor of separatrices for $\F$. Then
\[\bb_p(\F)=\var_p(\F,\B)+\mu_p(\F)-\sum_{B}a_B\mu_p(dF_{B},B)+\deg(\B)-1-\chi_p(\F)+\sum_{q\in\cl{I}_{p}(\F)}\xi^{2}_q(\F),\]
where ${F}_B$ is a balanced divisor of separatrices for $\F$ adapted to $B$.

Moreover, if $\F$ is non-dicritical and $C$ is the total union of separatrices of $\F$ then
\begin{equation*}
\bb_p(\F)=\cs_p(\F,C)+\tau_p(\F,C)-\tau_p(C)+\mu_p(\F)-\mu_p(C)-\chi_p(\F)+\sum_{q\in\cl{I}_{p}(\F)}\xi^{2}_q(\F).
\end{equation*}
\end{proposition}

%\par In the particular case of non-dicritical foliations we have:
%\begin{corollary}
%Let $\F$ be a singular foliation at $(\mathbb C^2,p)$. Assume that $\F$ is non-dicritical and $C$ is the total union of separatrices of $\F$. Then
%\begin{equation*}
%\bb_p(\F)=\cs_p(\F,C)+\tau_p(\F,C)-\tau_p(C)+\mu_p(\F)-\mu_p(C)-\chi_p(\F)+\sum_{q\in\cl{I}_{p}(\F)}\xi^{2}_q(\F).
%\end{equation*}
%\end{corollary}
\begin{proof}
Suppose that $\F$ is a singular foliation. By \cite[Theorem 5.2]{Fernandez} we get
\[\bb_p(\F)=\var_p(\F,\B)+\Delta_p(\F,\B)+\sum_{q\in\mathcal{I}_{p}(\F)}\xi^{2}_q(\F),\]
where the summation runs over all infinitely near points of $\F$ at $p$.
 The proof of the first equality follows applying Theorem \ref{Th:Delta-chi} to $\Delta_p(\F,\B)$. 

The proof of the non-dicritical case follows from the first part of the proposition (for $\B=C$ and $F_{C_i}=f$, for any $1\leq i\leq \ell$, where $f(x,y)=0$ is an equation of $C$), equality  \eqref{var_1},  Proposition \ref{obs_1} and equality \eqref{eq_mi}.
\end{proof}

\section{Bound for the Milnor sum of an algebraic curve}
\label{sect: bound}

Let $\F$ be a holomorphic foliation on the complex projective plane $\mathbb{P}^2$. The degree of $\F$ is the number of tangencies between $\F$ and a generic line. It is well-known that $\mathcal F$ has a least one singular point an there is a lot of
activity around the foliations of $\mathbb{P}^2$ having a unique singularity (for instance \cite{Cerveau-Deserti} and \cite{Castorena}).
 Let $C$ be an algebraic $\F$-invariant curve in $\mathbb{P}^2$. We say that $C$ is \textit{non-dicritical} if every singular point of $\F$ on $C$ is non-dicritical.

\par The classification of holomorphic foliations on $\mathbb{P}^2$ of degree $d$
has been of great interest in Algebraic Geometry (see for example the recent paper \cite{Castorena}). It is therefore interesting to obtain properties of the foliations in the projective plane.

\par As an application of our previous results, we have the following theorem:
\begin{theorem}\label{proyectivo}
Let $\F$ be a holomorphic foliation on $\mathbb{P}^2$ of degree $d$ leaving invariant a non-dicritical algebraic curve $C$ of degree $d_0$ such that for each $p\in\sing(\F)\cap C$, all local branches of $\sep_p(\F)$ are contained in $C$. Then
\[
\sum_{p\in C}\mu_p(C)= \left(\sum_{p\in \sing(\F)\cap C}[\mu_p(\F)-\chi_p(\F)]\right)+d^2_0-(d+2)d_0.
\]
In particular, 
\[
\sum_{p\in C}\mu_p(C)\leq \left(\sum_{p\in \sing(\F)\cap C}\mu_p(\F)\right)+d^2_0-(d+2)d_0.\]
Moreover, if $\sing(\F)\subset C$, then 
\[
\sum_{p\in C}\mu_p(C)\leq d^2-d(d_0-1)+(d_0-1)^2.
\]
\end{theorem}
\begin{proof}
Fix an arbitrary point $p\in\sing(\F)\cap C$. Then, by Corollary \ref{gsv=milnor}, we have 
\begin{eqnarray*}\label{eq_10}
\mu_p(C)=\mu_p(\F)-\chi_p(\F)-GSV_p(\F,C).
\end{eqnarray*} 
Since for each $p\in\sing(\F)\cap C$, all local branches of $\sep_p(\F)$ are contained in $C$, we get
equality $\displaystyle\sum_{p\in\sing(\F)\cap C}GSV_p(\F,C)=(d+2)d_0-d_0^2$ (cf. \cite[Proposition 4]{index}), and this implies 
\[
\sum_{p\in C}\mu_p(C)= \left(\sum_{p\in \sing(\F)\cap C}[\mu_p(\F)-\chi_p(\F)]\right)+d^2_0-(d+2)d_0.
\]
Note that, in particular, we get the inequality 
\[
\sum_{p\in C}\mu_p(C)\leq \left(\sum_{p\in \sing(\F)\cap C}\mu_p(\F)\right)+d^2_0-(d+2)d_0,
\]
since $\chi_p(\F)\geq 0$  by Proposition \ref{prop_chi} item (1).
On the other hand, if $\sing(\F)\subset C$, we use the equality $\displaystyle\sum_{p\in \sing(\F)}\mu_p(\F)=d^2+d+1$ (cf. \cite[page 28]{Brunella-libro}) to finish the proof.
\end{proof}

\begin{remark}
Observe that the hypothesis {\it all local branches of $\sep_p(\F)$ and all singularities of $\F$ are contained in $C$} in Theorem \ref{proyectivo} implies, by \cite[Proposition 6.1]{Fernandez}, that $\deg C=\deg \F +2$ and $\F$ is a logarithmic foliation.
\end{remark}

\noindent {\bf Acknowledgment} \\
The first-named author thanks Universidad de La Laguna for the hospitality during his visit. The authors thank the anonymous referee, whose remarks allowed them to improve the presentation. \\

\noindent{\bf Funding} The first author acknowledges support from CNPq Projeto Universal 408687/2023-1 "Geometria das Equa\c{c}\~oes Diferenciais Alg\'ebricas" and CNPq-Brazil PQ-306011/2023-9. The first-named and third-named authors were partially supported by the Pontificia Universidad Cat\'{o}lica del Per\'{u} project DFI 2024-PI1100. The second-named author  was supported by the grant PID2019-105896GB-I00 funded by MCIN/AEI/10.13039/501100011033.\\

\noindent{\bf Declarations}

\noindent{\bf Conflict of interest} The authors declare no Conflict of interest.\\

%The authors declare that they have no known competing financial interests or personal relationships that could have appeared to influence the work reported in this paper.\\

%\noindent{\bf Data availability}\\

 %Data sharing not applicable to this article as no datasets were generated
%or analysed during the current study.

\end{document}